\documentclass[sn-mathphys,Numbered]{sn-jnl}

\usepackage{graphicx}%
\usepackage{multirow}%
\usepackage{amsmath,amssymb,amsfonts}%
\usepackage{amsthm}%
\usepackage{mathrsfs}%
\usepackage[title]{appendix}%
\usepackage{xcolor}%
\usepackage{textcomp}%
\usepackage{manyfoot}%
\usepackage{booktabs}%
\usepackage{algorithm}%
\usepackage{algorithmicx}%
\usepackage{algpseudocode}%
\usepackage{listings}%
\usepackage{longtable}

\newtheorem{Theorem}{Theorem}
\newtheorem{Lemma}[Theorem]{Lemma}%
\newtheorem{Definition}[Theorem]{Definition}%

\newcommand{\be}{\begin{equation}}
	\newcommand{\ee}{\end{equation}}
\newcommand{\ba}{\begin{array}}
	\newcommand{\ea}{\end{array}}
\newcommand{\bea}{\begin{eqnarray}}
	\newcommand{\eea}{\end{eqnarray}}
\def\dfrac{\displaystyle\frac}

\def\dmin{\displaystyle\min}

\def\dfrac{\displaystyle\frac}

\def\fa{\forall}

\def\a{\alpha}

\def\d{\delta}
\def\D{\Delta}
\def\e{\epsilon}

\def\g{\gamma}

\def\l{\lambda}

\def\Om{\Omega}
\def\na{\nabla}
\def\m{\mathcal}
\def\st{{\rm s.t.}}

\begin{document}
	
	\title[Globally convergent SQP with least constraint violation for NSDP]{A globally convergent SQP-type method with least constraint violation for nonlinear semidefinite programming}
	
	
	\author*[1]{\fnm{Wenhao} \sur{Fu}}\email{wenhfu@usts.edu.cn}
	
	\author[2]{\fnm{Zhongwen} \sur{Chen}}\email{zwchen@suda.edu.cn}
	
	\affil*[1]{\orgdiv{School of Mathematical Sciences}, \orgname{Suzhou University of Science and Technology}, \orgaddress{\city{Suzhou}, \postcode{215009}, \state{Jiangsu}, \country{China}}}
	
	\affil[2]{\orgdiv{School of Mathematical Sciences}, \orgname{Soochow University}, \orgaddress{\city{Suzhou}, \postcode{215009}, \state{Jiangsu}, \country{China}}}
	
	
	\abstract{We present a globally convergent SQP-type method with the least constraint violation for nonlinear semidefinite programming. The proposed algorithm employs a two-phase strategy coupled with a line search technique. In the first phase, a subproblem based on a local model of infeasibility is formulated to determine a corrective step. In the second phase, a search direction that moves toward optimality is computed by minimizing a local model of the objective function. Importantly, regardless of the feasibility of the original problem, the iterative sequence generated by our proposed method converges to a Fritz-John point of a transformed problem, wherein the constraint violation is minimized. Numerical experiments have been conducted on various complex scenarios to demonstrate the effectiveness of our approach.}
	
	
	
	
	\maketitle
	
	\section{Introduction}
	\label{sect1}
	In this paper, we address the nonlinear semidefinite programming (NSDP) problem with equality constraints formulated as follows:
	\be
	\label{eq1.1}
	\ba{cl}
	\dmin_{x\in\m{R}^n} & f(x)\\
	{\rm s.t.}&h(x)=0,\\
	&G(x)\preceq0,
	\ea
	\ee
	where, $f:\m{R}^n\to\m{R}$, $h:\m{R}^n\to\m{R}^l$ and $G:\m{R}^n\to\m{S}^m$ are twice continuously differentiable functions, $\m{S}^m$ represents the space of $m\times m$ real symmetric matrices. The constraint $G(x)\preceq0$ indicates that $G(x)$ is a negative semidefinite matrix. We use $\m{S}_{--}^m$ and $\m{S}_-^m$ to denote the sets of negative definite and negative semidefinite symmetric matrices of dimension $m\times m$, respectively. Therefore, $G(x)\prec0$ $(G(x)\preceq0)$ is equal to $G(x)\in\m{S}_{--}^m$ $(G(x)\in\m{S}_-^m)$. The set $\m{S}_{++}^m$ and $\m{S}_+^m$ are defined similarly.
	
	In the past two decades, extensive research has been conducted in the field of NSDP, resulting in significant theoretical advancements (\cite{Forsgren00,Shapiro97,Sun06,Zhang20}) as well as practical algorithm developments (\cite{Mosheyev00,SunJ06,Wu13,Tseng02,Yamashita12,Yamashita20,Correa04,Gomez10,Zhao20}). NSDP has found wide applications in control, finance, eigenvalue problems, structural optimization, and other areas (\cite{Wolkowicz12,Apkarian03,Apkarian04,Kocvara04,Leibfritz06,Mostafa05,Tuan00,Weldeyesus15}).
	
	Many sequential quadratic programming (SQP) algorithms have been proposed to solve NSDP problems by extending the classical SQP-type method to $\m{S}_-^m$ cone. These globally convergent SQP-type algorithms typically find a Karush-Kuhn-Tucker (KKT) point of \eqref{eq1.1} or a feasible point where a certain constraint qualification fails. Various methods have been promoted in this regard, including the works of \cite{Gomez10,Chen15,Zhao16}.
	
	In the early 2000s, there was a growing interest in infeasibility detection for nonlinear programming, as infeasible problems often arise in practice. These problems can be caused by parameter mismatches or infeasible subproblems, such as in branch-and-bound methods. While existing SQP-type methods perform well in solving NSDP problems, it is worth noting that optimization models that simulate practical applications are only sometimes feasible. Therefore, when encountering an infeasible problem, it becomes necessary to identify the infeasible constraints while minimizing the objective function. It is important to study NSDP algorithms in more depth, exploring the possibility of taking a step forward instead of stopping at a stationary point with some measure of constraint violation. Recent work by Dai and Zhang (\cite{Dai22}) has proposed results for finding minimizers of the objective function with the least constraint violation. However, to the best of authors' knowledge, there are only a few publications available that discuss the infeasibility detection for NSDP and the relationship between KKT points and infeasible stationary points.
	
	Most studies on SQP-type methods for NSDP have primarily focused on finding KKT points. If there is no feasible point for \eqref{eq1.1}, some of the presented algorithms may terminate at an infeasible stationary point with some measure of constraint violation. Examples of such studies include \cite{Zhao16,Li19,Zhao20,Yamakawa22}. Existing literature on termination criteria for infeasible stationary points can be divided into two main categories: one describes the infeasibility of constraints, and the other deals with stabilization property for some measure of constraint violation. However, a key limitation of previous research is the need to address the relationship between the infeasible stationary point and the objective function. Therefore, several interesting and relevant problems remain to be addressed, including: 1) establishing a relationship between the infeasible stationary point and the objective function; 2) proposing a method that moves further instead of stopping at an infeasible stationary point; and 3) designing a globally convergent algorithm that converges to a Fritz-John (FJ) point instead of a feasible point when Robinson's constraint qualification fails, without relying on information from the objective function. This work aims to address these issues by seeking an algorithm with global convergence and finding points that minimize the objective function with the least constraint violation.
	
	In this paper, we present a novel SQP-type method to solving problem \eqref{eq1.1}. Our method addresses the limitations of existing techniques and offers several advantages. Specifically, it incorporates two sequences of subproblems to generate iterative solutions, ensuring global convergence to certain Fritz-John points without relying on any constraint qualifications.
	
	Key features of our proposed method include:
	\begin{itemize}
		\item Solvability of subproblems: Our approach guarantees that the subproblems are always solvable, eliminating the need for a restoration phase.
		
		\item Relationship between FJ points and infeasible stationary points: We establish a clear relationship between them, shedding light on their interplay in the optimization process.
		
		\item Improved termination criterion: We enhance the termination criterion for identifying infeasible stationary points, further refining the convergence properties of the method.
	\end{itemize}
	Overall, our technique can be viewed as a variant of SQP-type methods, offering distinct advantages over existing approaches. These contributions serve to advance the field of SQP and provide a valuable tool for solving the problem at hand.
	
	In the case where \eqref{eq1.1} is infeasible, we present our findings on the relationship between the infeasible stationary point and a related shifted problem, denoted as \eqref{eq1.2}. Analogous to the approach undertaken in a prior study by Dai and Zhang (\cite{Dai21}), the shifted problem presented in this study is formulated as follows:
	\be
	\label{eq1.2}
	\ba{ll}
	\dmin_{x\in\m{R}^n} & f(x)\\
	\st & h(x)=r-s,\\
	& G(x)\preceq tI_m,
	\ea
	\ee
	where $(r,s,t)\in\m{R}^l\times\m{R}^l\times\m{R}$ represents the least shift that locally solves a problem with respect to an infeasible stationary point $x$. More precisely, the following optimization problem, denoted as \eqref{eq2.2}, finds this least shift:
	\be
	\label{eq2.2}
	\ba{cl}
	\dmin_{x,r,s,t} & e_l^T(r+s)+t\\
	\st & h(x)=r-s, \\
	& G(x)\preceq tI_m, \\
	& r\ge0, \quad s\ge0, \quad t\ge0,
	\ea
	\ee
	where $e_l=(1,1,\cdots,1)$ represents a vector in $\m{R}^l$. The least shift characterizes a form of constraint violation and is similar to the definition given in \cite{Zhao20}. Consequently, we refer to problem \eqref{eq1.2} as an NSDP problem with the least constraint violation.
	
	Our proposed algorithm generates an infinite sequence of iterations $\{x_k\}$, with the accumulation point $x^*$ representing an FJ point of problem \eqref{eq1.2}. Unlike other SQP-type methods, the accumulation point $(r^*,s^*,t^*)$ of another sequence of iterations $\{(r_k,s_k,t_k)\}$ satisfies one of the following conditions:
	\begin{itemize}
		\item $(r^*,s^*,t^*)=0$, in which case $x^*$ corresponds to an FJ point of \eqref{eq1.1}.
		\item $(r^*,s^*,t^*)\ne0$, indicating that $x^*$ represents an infeasible stationary point of \eqref{eq1.1}.
	\end{itemize}
	
	The remaining part of this paper is structured as follows. In Section \ref{sect2}, preliminaries for this article are presented. The detailed algorithm is outlined in Section \ref{sect3}. Following this, Section \ref{sect4} undertakes a rigorous analysis of the algorithm?s well-definedness. Section \ref{sect5} discusses global convergence. Section \ref{sect6} empirically demonstrates the effectiveness of the algorithm via numerical experiments. The conclusion is reported in Section \ref{sect7}.
	
	\section{Preliminaries}
	\label{sect2}
	
	Throughout this article, we use simplified notations for the function values, such as $f_k:=f(x_k)$, $h_k:=h(x_k)$, $G_k:=G(x_k)$, $g_k:=g(x_k)$, where $g(x)=\na f(x)$ represents the gradient of $f(x)$. We define $Dh(x)$ as the $l\times n$ Jacobian matrix of $h(x)$, i.e., $Dh(x)^T=\na h(x)=(\na h_1(x),\cdots,\na h_l(x))$, where the superscript $ ^T$ stands for the transpose of a vector or a matrix. Given $x$, a linear operator $DG(x)$ is defined as
	\[
	DG(x)d:=\sum_{i=1}^n\dfrac{\partial G(x)}{\partial x_i}d_i, \quad \fa d\in\m{R}^n,
	\]
	where $x_i$ refers to the $i$th element of the vector $x$. The adjoint operator of $DG(x)$ is defined as:
	\[
	DG(x)^*Y:=\left(\left\langle\dfrac{\partial G(x)}{\partial x_1},Y\right\rangle,\cdots,\left\langle\dfrac{\partial G(x)}{\partial x_n},Y\right\rangle\right)^T, \quad \fa Y\in\m{S}^m,
	\]
	where $\langle A,B\rangle:=\text{tr}(A^TB)$ denotes the inner product of $A, B\in\m{S}^m$, and $\text{tr}(\cdot)$ represents the trace of a square matrix. We define $[a]_+:=\max\{0,a\}$ and $[a]_-:=\min\{0,a\}$. The Euclidean norm is denoted as $\|\cdot\|$ . More exactly, we use $\|\cdot\|_1$, $\|\cdot\|_2$ and $\|\cdot\|_\infty$ to represent the $\ell_1$-norm, $\ell_2$-norm and $\ell_\infty$-norm, respectively. For a given matrix $A\in\m{S}^m$, $\|\cdot\|_F$ refers to the Frobenius norm defined as $\|A\|_F:=\sqrt{\langle A,A\rangle}$. Let $\l_i(A)$ denote the $i$th eigenvalue in nonincreasing order. $\Pi_{\m{S}_-^m}(A)$ represents the orthogonal projection of $A$ onto $\m{S}_-^m$, which is defined as follows:
	\[
	\Pi_{\m{S}_-^m}(A)=P\text{diag}([\l_1(A)]_-,\cdots,[\l_m(A)]_-)P^T,
	\]
	where $P$ is the orthogonal matrix in the orthogonal decomposition
	\[
	A=P\text{diag}(\l_1(A),\cdots,\l_m(A))P^T.
	\]
	For vectors $a$ and $b$, their component-wise product is denoted as $a\circ b$, which yields a vector with entries $(a\circ b)_i=a_ib_i$.
	
	To derive the first-order optimality conditions for problem \eqref{eq1.1}, we introduce the Fritz-John (FJ) function $F:\m{R}^n\times\m{R}\times\m{R}^l\times\m{S}^m\to\m{R}$ for \eqref{eq1.1}, defined as
	\[
	F(x,\rho,\mu,Y):=\rho f(x)+\mu^Th(x)+\langle Y,G(x)\rangle,
	\]
	where $\rho\in\m{R}_+$ represents an objective multiplier, and $(\mu,Y)\in\m{R}^l\times\m{S}^m$ are multipliers associated with equality constraints and semidefinite constraints, respectively. The first-order optimality point of \eqref{eq1.2} satisfies the following conditions:
	\[
	\ba{c}
	\rho g(x)+Dh(x)^T\mu+DG(x)^*Y=0,\\
	h(x)=r-s,\\
	G(x)\preceq tI_m,\\
	\langle Y,G(x)-tI_m\rangle=0, \quad Y\succeq0.
	\ea
	\]
	Specifically, when $(r,s,t)=0$, the point $x$ also stands as a first-order optimality point of \eqref{eq1.1}.
	
	We define the measure of constraint violation $v(x)$ as
	\[
	v(x):=\|h(x)\|_1+[\l_1(G(x))]_+
	\]
	and we also denote $v_k:=v(x_k)$ for convenient. The infeasible stationary point, often mentioned, can be considered as a solution to the following least constraint violation problem
	\be
	\label{eq2.1}
	\ba{cl}
	\dmin_{x\in\m{R}^n} & v(x).
	\ea
	\ee
	It is important to note that such an infeasible stationary point is independent of the objective function $f(x)$. By introducing slack variables $(r,s,t)\in\m{R}^l\times\m{R}^l\times\m{R}$, the nonsmooth and nonconvex problem \eqref{eq2.1} can be equivalently transformed into the smooth and feasible problem \eqref{eq2.2}. Let $l(x,r,s,t,\mu,Y,\nu_r,\nu_s,\nu_t)$ be the Lagrangian function of \eqref{eq2.2} given by
	\begin{align*}
		&l(x,r,s,t,\mu,Y,\nu_r,\nu_s,\nu_t)\\
		=&e_l^T(r+s)+t+\mu^T(h(x)-r+s)+\langle Y,G(x)-tI_m\rangle-\nu_r^Tr-\nu_s^Ts-\nu_tt
	\end{align*}
	where  $(\nu_r,\nu_s,\nu_t)\in\mathcal{R}^l\times\mathcal{R}^l\times\mathcal{R}$ are the Lagrange multipliers associated with the inequality constraints. Under certain constraint qualification (CQ) conditions such as Robinson's CQ, the first-order optimality point of \eqref{eq2.2} satisfies that
	\[
	\ba{l}
	Dh(x)^T\mu+DG(x)^*Y=0,\\
	e_l-\mu-\nu_r=0,\quad
	e_l+\mu-\nu_s=0,\quad
	1-\langle Y,I_m\rangle-\nu_t=0,\\
	h(x)=r-s, \quad G(x)\preceq tI_m, \quad
	r\ge0, \quad s\ge0, \quad t\ge0,\\
	\langle Y,G(x)-tI_m\rangle=0, \quad
	\nu_r^Tr=0, \quad \nu_s^Ts=0, \quad \nu_t^Tt=0,\\
	Y\succeq0, \quad \nu_r\ge0, \quad \nu_s\ge0, \quad \nu_t\ge0.
	\ea
	\] 
	By utilizing the componentwise product ``$\circ$'' and setting $h(x)=[h(x)]_++[h(x)]_-$, the first-order optimality conditions for \eqref{eq2.2} can be further expressed as
	\be
	\label{eq2.3}
	\ba{l}
	\na_xF(x,0,\mu,Y)=Dh(x)^T\mu+DG(x)^*Y=0,\\
	(e_l-\mu)\circ[h(x)]_+=0, \quad (e_l+\mu)\circ[h(x)]_-=0,\\
	(1-\text{tr}(Y))[\l_1(G(x))]_+=0, \quad \langle Y,G(x)-[\l_1(G(x))]_+I_m\rangle=0,\\
	-e_l\le\mu\le e_l, \quad Y\succeq0, \quad \text{tr}(Y)\le1.
	\ea
	\ee
	Since problem \eqref{eq2.2} is an equivalent smooth formulation of \eqref{eq2.1}, we refer to \eqref{eq2.3} as a first-order optimality condition for \eqref{eq2.1} as well.
	
	In order to establish the global convergence of the algorithm proposed in the subsequent section, we consider the $\ell_1$ exact penalty function given by
	\be
	\label{eq2.4}
	P^\rho(x):=\rho f(x)+v(x).
	\ee
	In this paper, $\rho>0$ serves as both the penalty parameter and the objective multiplier. We define the linearization of $v(x)$ at the iterate point $x_k$ along a direction $d\in\m{R}^n$ as
	\[
	l_k^v(d):=\|h_k+Dh(x_k)d\|_1+[\l_1(G_k+DG(x_k)d)]_+
	\]
	and define the linearization of $P^\rho(x+d)$ at $x_k$ along $d$ as
	\[
	l_k^\rho(d):=\rho l_k^f(d)+l_k^v(d), \quad l_k^f(d):=f_k+g_k^Td.
	\]
	To quantify the linearized improvement in terms of $l_k^\rho(d)$, $l_k^v(d)$ and $l_k^f(d)$, respectively, with respect to $d$ compared to a zero step, we define
	\[
	\D l_k^\rho(d):=l_k^\rho(0)-l_k^\rho(d), \ \D l_k^v(d):=l_k^v(0)-l_k^v(d), \ \D l_k^f(d):=l_k^f(0)-l_k^f(d)=-g_k^Td.
	\]
	
	\section{Description of algorithm}
	\label{sect3}
	
	Now, we proceed to describe the detailed algorithm. Starting from the current iterate point $x_k$ with a positive definite matrix $B_k^{fea}\in\m{S}^n$, our first step is to compute a feasible direction by solving the following nonsmooth optimization problem:
	\be
	\label{eq3.1}
	\ba{cl}
	\dmin_{d\in\m{R}^n} & l_k^v(d)+\dfrac{1}{2}d^TB_k^{fea}d.
	\ea
	\ee
	To reformulate the above problem, we introduce slack variable $(r,s,t)\in\m{R}^l\times\m{R}^l\times\m{R}$, leading to the following equivalent quadratic semidefinite programming problem:
	\be
	\label{eq3.2}
	\ba{cl}
	\dmin_{d,r,s,t} & e_l^T(r+s)+t+\dfrac{1}{2}d^TB_k^{fea}d\\
	{\rm s.t.}  & h(x_k)+Dh(x_k)d=r-s,\\
	& G(x_k)+DG(x_k)d\preceq tI_m,\\
	& r\ge0, \quad s\ge0, \quad t\ge0.
	\ea
	\ee
	Let $(d_k^{fea},r_k,s_k,t_k)$ be the solution obtained. It is easy to confirm that $d_k^{fea}$ is unique and that $l_k^v(d_k^{fea})=e_l^T(r_k+s_k)+t_k$. Set $(\bar\mu_{k+1},\bar Y_{k+1})\in\m{R}^l\times\m{S}^m$ as the Lagrange multipliers associated with the equality and semidefinite constraint in \eqref{eq3.2}. We define the feasibility residual $R_k^{fea}:=R^{fea}(x_k,\bar\mu_{k+1},\bar Y_{k+1})$, where
	\begin{align*}
		R^{fea}(x,\mu,Y)
		&=\|\na_xF(x,0,\mu,Y)\|_{\infty}\\
		&\quad +\|(e_l-\mu)\circ[h(x)]_+\|_\infty+\|(e_l+\mu)\circ[h(x)]_-\|_\infty\\
		&\quad +\lvert1-\text{tr}(Y)\rvert[\l_1(G(x))]_++\|Y(G(x)-[\l_1(G(x))]_+I_m)\|_F.
	\end{align*}
	One can see that the last condition of \eqref{eq2.3} is always satisfied for $(\bar\mu_{k+1},\bar Y_{k+1})$ since they are Lagrange multipliers. If $R^{fea}(x,\mu,Y)=0$, then $(x,\mu,Y)$ satisfies the first-order optimality condition \eqref{eq2.3}, implying that $x$ is a stationary point of \eqref{eq2.1}.
	
	Given some positive definite matrix $B_k\in\m{S}^n$ which is normally different from $B_k^{fea}$ in \eqref{eq3.1}, we compute $d_k$ as a search direction by solving a quadratic semidefinite programming problem
	\be
	\label{eq3.3}
	\ba{ll}
	\dmin_{d}& \rho_kg_k^Td+\dfrac{1}{2}d^TB_kd\\
	{\rm s.t.} & h_k+Dh(x_k)d=r_k-s_k,\\
	& G_k+DG(x_k)d\preceq t_kI_m.
	\ea
	\ee
	Set $(\hat\mu_{k+1},\hat Y_{k+1})\in\m{R}^l\times\m{S}^m$ as the Lagrange multipliers for the equality constraints and semidefinite constraint in \eqref{eq3.3}. Denote the optimality residual by $R_k^{opt}:=R^{opt}(x_k,\rho_k,\hat\mu_{k+1},\hat Y_{k+1})$, where
	\[
	R^{opt}(x,\rho,\mu,Y)=\|\na_xF(x,\rho,\mu,Y)\|_\infty+\|YG(x)\|_F.
	\]
	If $v(x)=0$ and $R^{opt}(x,\rho,\mu,Y)=0$, then $x$ is a stationary point of \eqref{eq1.1}.
	
	Let us direct our attention towards the formulation of our line search methodology. To this end, we employ the $\ell_1$ exact penalty
	function denoted as $P^\rho(x)$ in equation \eqref{eq2.4}, with $\rho=\rho_k$ signifying the parameter's reliance on the $k$-th iteration step. The consequent reduction within the linear model of
	$P^{\rho_k}(x)$,  induced by the selected search direction $d_k$, takes the form of
	\[
	\D l_k^{\rho_k}(d_k)=-\rho_kg_k^Td_k+\D l_k^v(d_k).
	\]
	Considering the condition
	\be
	\label{eq3.4}
	\rho_k(\|\bar\mu_{k+1}\|_\infty+\text{tr}(\bar Y_{k+1}))>1 \quad \text{or} \quad
	\rho_k(\|\hat\mu_{k+1}\|_\infty+\text{tr}(\hat Y_{k+1}))>1,
	\ee
	we proceed to update the parameter $\rho_k$ to $\rho_{k+1}$ through that
	\be
	\label{eq3.5}
	\rho'_k=\left\{\ba{ll}
	\min\left\{\delta\rho_k,\ \dfrac{1-\e}{\|\bar\mu_{k+1}\|_\infty+\text{tr}(\bar
		Y_{k+1})+\|\hat\mu_{k+1}\|_\infty+\text{tr}(\hat Y_{k+1})}\right\}, & \text{if \eqref{eq3.4}
		holds,}\\
	\rho_k, & \text{otherwise,}
	\ea\right.
	\ee
	and
	\be
	\label{eq3.6}
	\rho_{k+1}=\left\{\ba{ll}
	\min\{\delta\rho'_k,\ \zeta_k\}, & \text{if $\D l_k^{\rho'_k}(d_k)<\e\D l_k^v(d_k)$,}\\
	\rho'_k, & \text{if $\D l_k^{\rho'_k}(d_k)\ge\e\D l_k^v(d_k)$,}
	\ea\right.
	\ee
	where $\delta\in(0,1)$ and $\e\in(0,1)$ are constants, and
	\[
	\zeta_k:=\dfrac{(1-\e)\D l_k^v(d_k)}{g_k^Td_k+0.5d_k^TB_kd_k}.
	\]
	It is worth noting that the protocol for updating  $\rho_k$ presented here extends the approach outlined in \cite{Burke14} with minor refinements. Upon the successful update of $\rho_{k+1}$, the subsequent step entails an Armijo line search executed along the direction $d_k$. In more precise terms, let $\a:=\a_k$ represent the first member of the sequence $\{1,\g,\g^2,\cdots\}$, $0<\g<1$, satisfying the condition
	\be
	\label{eq3.7}
	P^{\rho_{k+1}}(x_k+\a d_k)-P^{\rho_{k+1}}(x_k)\le-\eta\a\D l_k^{\rho_{k+1}}(d_k), \quad
	\eta\in(0,1).
	\ee
	Finally, upon setting $x_{k+1}=x_k+\a_kd_k$,  the algorithm advances to the ensuing iteration.
	
	With the groundwork thus meticulously established, we are now ready to outline the algorithmic framework that systematically addresses the NSDP problem \eqref{eq1.1}.
	
	\begin{algorithm}[H]
		\caption{Sequential quadratic programming with the least constraint violation}
		\label{alg3.1}
		\begin{algorithmic}
			\State {\bf Initialization:} Given $x_0\in\m{R}^n, \ B_0^f\in\m{S}_{++}^n, \ B_0\in\m{S}_{++}^n, 0<\e,\d,\eta,\gamma<1, \ \rho_0>0, \ k=0, \ nmax>0$.
			\While {$k\le nmax$}
			\State Solve \eqref{eq3.2} to get $(d_k^{fea},r_k,s_k,t_k)$, let $(\bar\mu_{k+1},\bar
			Y_{k+1})$ be its corresponding Lagrange multipliers.
			\State Solve \eqref{eq3.3} to get $d_k$, let $(\hat\mu_{k+1},\hat Y_{k+1})$ be its
			corresponding Lagrange multipliers.
			\If{$d_k=0$}
			\If{$v_k=0$}
			\State $x_k$ is a KKT point for \eqref{eq1.1}. {\bf Break}.
			\Else
			\State $x_k$ is an infeasible stationary point for \eqref{eq1.1}.
			\State $x_k$ is a KKT point for problem \eqref{eq1.2}. {\bf Break}.
			\EndIf
			\EndIf
			\State Update $\rho_k$ to $\rho_{k+1}$ by  \eqref{eq3.5} and \eqref{eq3.6}.
			\State Let $\a_{k,0}=1$, $i:=0$.
			\While {$i\ge0$}
			\If {\eqref{eq3.7} holds for $\a_{k,i}$}
			\State $\a_k=\a_{k,i}$, \ $x_{k+1}=x_k+\a_kd_k$. {\bf Break}.
			\Else
			\State $\a_{k,i+1}=\gamma\a_{k,i}$, \ $i:=i+1$.
			\EndIf
			\EndWhile \quad (for $i$)
			\State Update $(B_k^{fea},B_k)$ to $(B_{k+1}^{f},B_{k+1})$, $k:=k+1$.
			\EndWhile \quad (for $k$)
		\end{algorithmic}
	\end{algorithm}
	
	\section{Well-definedness}
	\label{sect4}
	
	Henceforth, shall delve into the analysis of the well-definedness of Algorithm \ref{alg3.1}. In doing so, we rely on the following assumptions regarding the sequence of iterates $\{x_k\}$, and the matrices $\{(B_k^{fea},B_k)\}$ generated by the algorithm.
	
	{\noindent\rm \bf Assumption A}\\
	(A1) $f(x)$, $h(x)$ and $G(x)$ are twice continuously differentiable.\\
	(A2) There exists a convex and compact set $\Om\subseteq\m{R}^n$ such that $x_k\in\Om$ for all
	$k$.\\
	(A3) The matrix sequence $\{B_k^{fea}\}$ and $\{B_k\}$ are uniformly positive definite and bounded
	above for all $k$, i.e., there exist two constants $0<b_1<b_2$ such that, for all $k$ and
	$d\in\m{R}^n$,
	$$
	b_1\|d\|^2\leq d^TB_k^{fea}d\leq b_2\|d\|^2, \quad
	b_1\|d\|^2\leq d^TB_kd\leq b_2\|d\|^2.
	$$
	
	Several preliminary results are essential for the underpinning of our methodology. The subsequent lemma provides insights into the reductions of both $l_k^v(d)$ and $l_k^\rho(d)$.
	
	\begin{Lemma}
		\label{lem4.1}
		The reductions of $l_k^v(\cdot)$ and $l_k^\rho(\cdot)$ from $0$ to $d$ satisfy
		\be
		\label{eq4.1}
		v'(x_k;d)\le-\D l_k^v(d), \quad (P^\rho)'(x_k;d)\le-\D l_k^\rho(d),
		\ee
		where $v'(x;d)$ and $(P^\rho)'(x;d)$ represent the directional derivatives of $v(\cdot)$ and $P^\rho(\cdot)$ at $x_k$ along a direction $d\in\m{R}^n$, respectively.
	\end{Lemma}
	
	\begin{proof}
		By the definition of $v(x)$, $l_k^v(d)$, and Assumption A1,
		\begin{align*}
			v'(x_k;d)=&\lim_{t\to0^+}\frac{v(x_k+td)-v(x_k)}{t}\\
			=&\lim_{t\to0^+}\frac{\|h_k+tDh(x_k)d\|_1+[\lambda_1(G_k+tDG(x_k)d)]_++o(t)-v(x_k)}{t}\\
			=&\lim_{t\to0^+}\frac{\|h_k+tDh(x_k)d\|_1+[\lambda_1(G_k+tDG(x_k)d)]_+-l_k^v(0)}{t}+o(1)\\
			=&(l_k^v)'(0;d).
		\end{align*}
		Moreover, due to the combined convexity and monotonically non-decreasing nature of the function $[\cdot]_+$, along with the convexity of $\lambda_1(\cdot)$,  we establish the following inequality:
		\[
		[\lambda_1(\theta G_1+(1-\theta)G_2)]_+\le[\theta\lambda_1(G_1)+(1-\theta)\lambda_1(G_2)]_+\le\theta[\lambda_1(G_1)]_++(1-\theta)[\lambda_1(G_2)]_+,
		\]
		which follows that the mapping $[\lambda_1(\cdot)]_+:\m{S}^m\to\m{R}$ is a convex one. When coupled with the convexity of $\|\cdot\|_1$, it becomes evident that $l_k^v(d)$ exhibits convexity in relation to $d\in\m{R}^n$.  As a consequence, we deduce that
		\begin{align*}
			(l_k^v)'(0;d)=&\lim_{t\to0^+}\frac{l_k^v(td)-l_k^v(0)}{t}\\
			\le&\lim_{t\to0^+}\frac{(1-t)l_k^v(0)+tl_k^v(d)-l_k^v(0)}{t}\\
			=&-l_k^v(0)+l_k^v(d)\\
			=&-\D l_k^v(d).
		\end{align*}
		In conclusion, the remaining proof can be readily derived by considering $v'(x_k;d)=(l_k^v)'(0;d)$, $(P^\rho)'(x_k;d)=\rho g_k^Td+v'(x_k;d)$, and the notation $\D l_k^\rho(d)$, which is defined as $\D l_k^{\rho}(d)=-\rho g_k^Td+\D l_k^v(d)$ in Section \ref{sect2}.
	\end{proof}
	
	The following results in Lemma \ref{lem4.2} are widely recognized within the realm of nonlinear semidefinite programming, as evidenced by existing literature such as \cite{Zhao20}. While the fundamental proofs align with established understanding, it's worth noting that certain aspects of the detailed derivations may exhibit distinct nuances.
	
	\begin{Lemma}
		\label{lem4.2}
		We have the followings:
		
		(a) The subproblem \eqref{eq3.2} is both feasible and yields a unique solution component $d_k^{fea}$.
		
		(b) $\D l_k^v(d_k^{fea})\ge0$, where the equality holds if and only if $d_k^{fea}=0$.
		
		(c) $d_k^{fea}=0$ if and only if $x_k$ is a stationary point for $v(\cdot)$.
		
		(d) $d_k^{fea}=0$ if and only if $(x_k,\bar\mu_{k+1},\bar Y_{k+1})$ satisfies first-order optimality condition \eqref{eq2.3}.
	\end{Lemma}
	
	\begin{proof}
		(a) It is clear that the point $(0,[h_k]_+,[h_k]_-,[\l_1(G_k)]_+)$ represents a feasible solution of \eqref{eq3.2}. Notably, the optimization problem \eqref{eq3.2} demonstrates equivalence to \eqref{eq3.1}, which itself constitutes a strictly convex programming problem. This inherent convexity ensures the uniqueness of the solution to \eqref{eq3.1}, consequently leading to the uniqueness of $d_k^{fea}$.
		
		(b) Given that $d_k^{fea}$ represents an optimal point, and recognizing that $d=0$ is a feasible solution of \eqref{eq3.1}, it follows that
		\[
		l_k^v(d_k^{fea})+\frac{1}{2}(d_k^{fea})^TB_k^{fea}d_k^{fea}\le l_k^v(0).
		\]
		Consequently, we deduce that
		\be
		\label{eq4.2}
		\D l_k^v(d_k^{fea})=l_k^v(0)-l_k^v(d_k^{fea})\ge\frac{1}{2}(d_k^{fea})^TB_k^{fea}d_k^{fea}\ge0.
		\ee
		On the one hand, if $d_k^{fea}=0$, it is apparent from the definition of $\D l_k^v(d)$ that $\D l_k^v(d_k^{fea})=l_k^v(0)-l_k^v(d_k^{fea})=0$. On the other hand, if the equality $\D l_k^v(d_k^{fea})=0$ holds, then it follows from \eqref{eq4.2} that $(d_k^{fea})^TB_k^{fea}d_k^{fea}=0$. By the positive
		definiteness of $B_k^{fea}$, it follows that $d_k^{fea}=0$.
		
		(c) It follows from the proof in Lemma \ref{lem4.1} that
		$v'(x_k;d)=(l_k^v)'(0;d)$. Therefore, $x_k$ is a stationary point for $v(\cdot)$ if and only if $0$
		is a stationary point for $l_k^v(\cdot)$, which is a global minimizer of $\min_dl_k^v(d)$.
		Thus, there are two scenarios to consider.
		
		Firstly, if $x_k$ indeed constitutes a stationary point for $v(\cdot)$, then $l_k^v(0)\le
		l_k^v(d_k^{fea})$. Using the relation established in \eqref{eq4.2}, we can observe that  $\D l_k^v(d_k^{fea})=l_k^v(0)-l_k^v(d_k^{fea})=0$. According to (b), it can be deduced that $d_k^{fea}=0$ holds.
		
		Conversely, the condition $\D l_k^v(d_k^{fea})=0$ implies, based on (b), that $d_k^{fea}=0$, and $0$ is a global
		minimizer of \eqref{eq3.2}. Hence, we have:
		\[
		0\le\left(l_k^v(d)+\frac{1}{2}d^TB_k^{fea}d\right)'(0;d)=(l_k^v)'(0;d)=v'(x_k;d), \quad \fa
		d\in\m{R}^n,
		\]
		leading to the conclusion that $x_k$ is a stationary point for $v(\cdot)$.
		
		(d) The proof can be established through \eqref{eq2.3} and the utilization of the first-order optimality condition of \eqref{eq3.2}, which can be expressed as follows:
		\be
		\label{eq4.3}
		\ba{c}
		B_k^{fea}d_k^{fea}+Dh(x_k)^T\bar\mu_{k+1}+DG(x_k)^*\bar Y_{k+1}=0,\\
		(e_l-\bar\mu_{k+1})\circ[h_k+Dh(x_k)d_k^{fea}]_+=0,\\
		(e_l+\bar\mu_{k+1})\circ[h_k+Dh(x_k)d_k^{fea}]_-=0,\\
		(1-\text{tr}(\bar Y_{k+1}))[\l_1(G_k+DG(x_k)d_k^{fea})]_+=0,\\
		\langle\bar Y_{k+1},G_k+DG(x_k)d_k^{fea}-[\l_1(G_k+DG(x_k)d_k^{fea})]_+I_m\rangle=0,\\
		-e_l\le\bar\mu_{k+1}\le e_l, \quad \bar Y_{k+1}\succeq0, \quad \text{tr}(\bar Y_{k+1})\le1.
		\ea
		\ee
		~
	\end{proof}
	
	The subsequent result presents a property concerning the search direction $d_k$ derived from problem \eqref{eq3.3}.
	
	\begin{Lemma}
		\label{lem4.3}
		If $\rho_k>0$ and $v_k=0$, then $(x_k,\hat\mu_{k+1}/\rho_k,\hat Y_{k+1}/\rho_k)$ is a KKT
		point for \eqref{eq1.1} if and only if $d_k=0$.
	\end{Lemma}
	
	\begin{proof}
		In the case where $v_k=0$, it is evidence that $l_k^v(d_k^{fea})=0$. Consequently, this implies that $(r_k,s_k,t_k)=0$, leading to the fulfillment of the KKT condition for $(d_k,\hat\mu_{k+1},\hat Y_{k+1})$:
		\be
		\label{eq4.4}
		\ba{c}
		\rho_kg_k+B_kd_k+Dh(x_k)^T\hat\mu_{k+1}+DG(x_k)^*\hat Y_{k+1}=0,\\
		h_k+Dh(x_k)d_k=0,\\
		G_k+DG(x_k)d_k\preceq0,\\
		\langle\hat Y_{k+1},G_k+DG(x_k)d_k\rangle=0, \quad \hat Y_{k+1}\succeq0.
		\ea
		\ee
		Moreover, since $\rho_k>0$, the point $(x_k,\hat\mu_{k+1}/\rho_k,\hat Y_{k+1}/\rho_k)$ satisfies that
		\be
		\label{eq4.5}
		\ba{c}
		g_k+(B_k/\rho_k)d_k+Dh(x_k)^T(\hat\mu_{k+1}/\rho_k)+DG(x_k)^*(\hat Y_{k+1}/\rho_k)=0,\\
		h_k+Dh(x_k)d_k=0,\\
		G_k+DG(x_k)d_k\preceq0,\\
		\langle\hat Y_{k+1}/\rho_k,G_k+DG(x_k)d_k\rangle=0, \quad \hat Y_{k+1}/\rho_k\succeq0.
		\ea
		\ee
		On the one hand, if $d_k=0$, then \eqref{eq4.5} serves as the KKT condition for problem \eqref{eq1.1}.
		Conversely, if $(x_k,\hat\mu_{k+1}/\rho_k,\hat Y_{k+1}/\rho_k)$ is a KKT point for
		\eqref{eq1.1}, then $(0,\hat\mu_{k+1},\hat Y_{k+1}))$ satisfies the KKT condition
		\eqref{eq4.4}. Given that \eqref{eq3.3} represents a strictly convex programming problem, it can be deduced that $d_k=0$.
	\end{proof}
	
	The following lemma demonstrates that the line search procedure of Algorithm \ref{alg3.1} terminates after reducing the value for a finite number of times.
	
	\begin{Lemma}
		\label{lem4.4}
		If Algorithm \ref{alg3.1} does not terminate at $x_k$, then the following statements hold true:
		
		(a) If $\rho_k>0$, then $\rho_{k+1}>0$
		and
		\be
		\label{eq4.6}
		\D l_k^{\rho_{k+1}}(d_k)\ge\e\D l_k^v(d_k)\ge\e\D l_k^v(d_k^{fea})\ge0.
		\ee
		
		(b) $\D l_k^{\rho_{k+1}}(d_k)>0$, and the line search procedure terminates finitely at $\a_k\in(0,1]$.
	\end{Lemma}
	
	\begin{proof}
		(a) We begin by demonstrating that $\rho'_k>0$. Since both $\bar Y_{k+1}$ and $\hat
		Y_{k+1}$ are Lagrange multipliers associated with the semidefinite constraints, we know that $\bar Y_{k+1},\hat Y_{k+1}$ are positive semidefinite matrices, and the trace of each matrix is nonnegative. If the condition
		\[
		\|\bar\mu_{k+1}\|_\infty+\text{tr}(\bar Y_{k+1})+\|\hat\mu_{k+1}\|_\infty+\text{tr}(\hat
		Y_{k+1})=0
		\]
		is satisfied, then \eqref{eq3.4} is violated, leading to $\rho'_k=\rho_k>0$. Alternatively, either
		$\rho'_k=\rho_k>0$ or $\rho_k$ is updated in a manner such that
		\[
		\rho_k'=\min\left\{\delta\rho_k,\ \dfrac{1-\e}{\|\bar\mu_{k+1}\|_\infty+\text{tr}(\bar
			Y_{k+1})+\|\hat\mu_{k+1}\|_\infty+\text{tr}(\hat Y_{k+1})}\right\}>0.
		\]
		
		Considering that $d_k$ solves \eqref{eq3.3}, and $(r_k,s_k,t_k)$ is a solution to \eqref{eq3.2}, we establish the following inequalities:
		\[
		l_k^v(d_k)\le\|r_k-s_k\|_1+t_k\le e_l^T(r_k+s_k)+t_k=l_k^v(d_k^{fea}).
		\]
		Furthermore, utilizing \eqref{eq4.2}, we have that
		\[
		\D l_k^v(d_k)\ge\D l_k^v(d_k^{fea})\ge0.
		\]
		Consequently, all that remains to prove that
		\be
		\label{eq4.7}
		\rho_{k+1}>0, \quad \D l_k^{\rho_{k+1}}(d_k)\ge\e\D l_k^v(d_k).
		\ee
		We will consider two cases: when $\D l_k^v(d_k)=0$ and when $\D l_k^v(d_k)>0$.
		
		In the instance where $\D l_k^v(d_k)=0$, it follows that $\D l_k^v(d_k^{fea})=0$, which, by Lemma \ref{lem4.2}, implies
		$d_k^{fea}=0$. Since the algorithm does not terminate at $x_k$, it is evident that $d_k\neq0$. According to Assumption A3, $d_k^TB_kd_k>0$.
		Since $d_k^{fea}=0$ is a feasible point for \eqref{eq3.3}, we have that
		\[
		\rho_kg_k^Td_k<\rho_kg_k^Td_k+\frac{1}{2}d_k^TB_kd_k\le0.
		\]
		Hence, we deduce $g_k^Td_k<0$ and thus that
		\[
		\D l_k^{\rho_k'}(d_k)
		=\rho_k'\D l_k^f(d_k)+\D l_k^v(d_k)
		=-\rho_k'g_k^Td_k
		>0=\e\D l_k^v(d_k),
		\]
		which follows by \eqref{eq3.6} that $\rho_{k+1}=\rho_k'$, confirming the validity of \eqref{eq4.7}.
		
		In the scenario where $\D l_k^v(d_k)>0$, we distinguish between two sub-cases. If $\D l_k^{\rho_k'}(d_k)\ge\e\D
		l_k^v(d_k)$, then $\rho_{k+1}=\rho_k'$, again confirming \eqref{eq4.7}. However, if $\D l_k^{\rho_k'}(d_k)<\e\D
		l_k^v(d_k)$, then the following inequality holds:
		\[
		-\rho_k'g_k^Td_k+\D l_k^v(d_k)<\e\D l_k^v(d_k),
		\]
		which subsequently leads to
		\[
		g_k^Td_k>\frac{1-\e}{\rho_k'}\D l_k^v(d_k)>0.
		\]
		As a result, $\zeta_k>0$ and thus that $\rho_{k+1}=\min\{\d\rho'_k,\zeta_k\}>0$. We conclude by $g_k^Td_k>0$ and $\rho_{k+1}\le\zeta_k$ that
		\begin{align*}
			\D l_k^{\rho_{k+1}}(d_k)
			&=-\rho_{k+1}g_k^Td_k+\D l_k^v(d_k)\\
			&\ge-\zeta_kg_k^Td_k+\D l_k^v(d_k)\\
			&=-\dfrac{(1-\e)\D l_k^v(d_k)}{g_k^Td_k+0.5d_k^TB_kd_k}g_k^Td_k+\D l_k^v(d_k)\\
			&\ge-\dfrac{(1-\e)\D l_k^v(d_k)}{g_k^Td_k}g_k^Td_k+\D l_k^v(d_k)\\
			&=\e\D l_k^v(d_k),
		\end{align*}
		which proves \eqref{eq4.7}.
		
		(b) We begin by asserting that $\D l_k^{\rho_{k+1}}(d_k)>0$. Indeed, by (a), $\D
		l_k^{\rho_{k+1}}(d_k)=0$ only if $\D l_k^v(d_k)=\D l_k^v(d_k^{fea})=0$. This implies that $d_k^{fea}=0$.
		Consequently, $d=0$ becomes a feasible point for \eqref{eq3.3}. As $d_k\neq0$, it follows that
		\[
		\rho_kg_k^Td_k<\rho_kg_k^Td_k+\frac12d_k^TB_kd_k\le0,
		\]
		which leads to $g_k^Td_k<0$ and thus that $\D l_k^{\rho_{k+1}}(d_k)=-\rho_{k+1}g_k^Td_k>0$. We have established that $\D
		l_k^{\rho_{k+1}}(d_k)>0$. By the definition of $(P^\rho)'(x_k;d))$ and \eqref{eq4.1}, it follows that
		\[
		\lim_{\a\to0^+}\dfrac{P^{\rho_{k+1}}(x_k+\a d_k)-P^{\rho_{k+1}}(x_k)}{\a}=
		(P^{\rho_{k+1}})'(x_k;d_k)\le-\D l_k^{\rho_{k+1}}(d_k)<0.
		\]
		Taking into consideration the continuity of $P^\rho(x)$ and the definition of limit, for $\eta\in(0,1)$, there exists $\bar\a_k\in(0,1]$
		sufficiently small such that
		\[
		P^{\rho_{k+1}}(x_k+\a d_k)-P^{\rho_{k+1}}(x_k)
		\le\eta\a(P^{\rho_{k+1}})'(x_k;d_k)
		\le-\eta\a\D l_k^{\rho_{k+1}}(d_k)
		\]
		holds for all $\a\in(0,\bar\a_k]$.
	\end{proof}
	
	Since the subproblems \eqref{eq3.2} and \eqref{eq3.3} are always feasible, and based on the conclusion drawn from Lemma \ref{lem4.4} that the line search procedure terminates in a finite of times, we are able to establish that Algorithm \ref{alg3.1} is well defined.
	
	\begin{Theorem}
		\label{thm4.5}
		Algorithm \ref{alg3.1} exhibits two possible outcomes: either it terminates finitely, or it generates an infinite sequence of iterations $\{(x_k,\a_k,d_k^{fea},d_k)\}$ along with associated multipliers $\{(\rho_k,\bar\mu_{k+1},\bar
		Y_{k+1},\hat\mu_{k+1},\hat Y_{k+1})\}$ satisfying
		\[
		\rho_k>0, \quad -e\le\bar\mu_{k+1}\le e, \quad \bar Y_{k+1}\succeq0, \quad {\rm tr}(\bar
		Y_{k+1})\le1, \quad \hat Y_{k+1}\succeq0.
		\]
	\end{Theorem}
	
	\section{Global convergence}
	\label{sect5}
	
	In this section, we establish the global convergence results for Algorithm \ref{alg3.1}. Firstly, we introduce Robinson's constraint qualification.
	
	\begin{Definition}
		\label{def5.1}
		A feasible point $x^*$ of the problem \eqref{eq1.1} satisfies Robinson's constraint qualification if and only if $Dh(x^*)$ has full row rank, and there exists a unit vector $\bar d\in\m{R}^n$ such that
		\[
		Dh(x^*)\bar d=0, \quad G(x^*)+DG(x^*)\bar d\prec0.
		\]
	\end{Definition}
	
	The following lemma demonstrates that the solutions for the subproblem \eqref{eq3.1} and \eqref{eq3.3} are both bounded.
	
	\begin{Lemma}
		\label{lem5.2}
		The sequences $\{d_k^{fea}\}$ and $\{d_k\}$ are both bounded.
	\end{Lemma}
	
	\begin{proof}
		By Assumption A1 and A2, there exists a constant $v_{\max}>0$ such that $v_k\le v_{\max}$ holds for
		all $k$. Suppose, by contradiction, that the sequence $\{d_k^{fea}\}$ is unbounded. Then, there
		exists an infinite index set $\m{K}_1$ such that
		\[
		\|d_k^{fea}\|^2>\dfrac{2v_{\max}}{b_1}, \quad \fa k\in\m{K}_1,
		\]
		where $b_1$ is defined in Assumption A3. However, since $(0,[h_k]_+,[h_k]_-,[\l_1(G_k)]_+)$ is a feasible point for \eqref{eq2.3}, we deduce that for $k\in\m{K}_1$,
		\begin{align*}
		v_k=l_k^v(0)
		&\ge l_k^v(d_k^{fea})+\dfrac{1}{2}(d_k^{fea})^TB_k^{fea}d_k^{fea}\\
		&\ge\dfrac{1}{2}(d_k^{fea})^TB_k^{fea}d_k^{fea}\\
		&\ge\dfrac{1}{2}b_1\|d_k^{fea}\|^2\\
		&>v_{\max}\\
		&\ge v_k.
		\end{align*}
		This contradiction arises due to the inconsistency between the derived inequality and $v_k\le v_{\max}$. As a result, we conclude that the sequence $\{d_k^{fea}\}$ is bounded.
		
		We suppose, by contradiction, that there exists an infinite index set $\m{K}_2$ such that
		\[
		\|d_k\|\ge\max\left\{1+\dfrac{8\rho_0\|g_k\|}{b_1},\sqrt{\dfrac{2b_2}{b_1}}\|d_k^{fea}\|\right\},
		\quad k\in\m{K}_2.
		\]
		Then,
		\[
		\rho_0\|g_k\|\|d_k\|<\dfrac{1}{8}b_1\|d_k\|^2,
		\]
		\[
		\rho_0\|g_k\|\|d_k^{fea}\|<\dfrac{1}{8}b_1\sqrt{\dfrac{b_1}{2b_2}}\|d_k\|^2\le\dfrac{1}{8}b_1\|d_k\|^2,
		\]
		\[
		\dfrac{1}{2}b_2\|d_k^{fea}\|^2\le\dfrac{1}{4}b_1\|d_k\|^2.
		\]
		Thus
		\begin{align*}
			&\quad-\rho_kg_k^Td_k+\rho_kg_k^Td_k^{fea}+\dfrac{1}{2}(d_k^{fea})^TB_k^{fea}d_k^{fea}\\
			&\le\rho_0\|g_k\|\|d_k\|+\rho_0\|g_k\|\|d_k^{fea}\|+\dfrac{1}{2}b_2\|d_k^{fea}\|^2\\
			&<\dfrac{1}{8}b_1\|d_k\|^2+\dfrac{1}{8}b_1\|d_k\|^2+\dfrac{1}{4}b_1\|d_k\|^2\\
			&=\dfrac{1}{2}b_1\|d_k\|^2\\
			&\le\dfrac{1}{2}(d_k)^TB_kd_k,
		\end{align*}
		i.e.,
		\[
		\rho_kg_k^Td_k^{fea}+\dfrac{1}{2}(d_k^{fea})^TB_k^{fea}d_k^{fea}
		<\rho_kg_k^Td_k+\dfrac{1}{2}(d_k)^TB_kd_k.
		\]
		This contradicts the fact that $d_k^{fea}$ is a feasible point for \eqref{eq3.3}, while $d_k$ is a global minimizer for it. Hence, by the boundedness of $d_k^{fea}$, we deduce that $d_k$ is bounded by
		\[
		\|d_k\|\le\sup_k\max\left\{1+\dfrac{8\rho_0\|g_k\|}{b_1},\sqrt{\dfrac{2b_2}{b_1}}\|d_k^{fea}\|\right\}.
		\]
		In conclusion, by Assumption A1 and A2, the boundedness of $\{\|g_k\|\}$ implies the boundedness of $\{\|d_k\|\}$.
	\end{proof}
	
	The following lemma provides a lower bound for $\a_k$ for each $k$.
	
	\begin{Lemma}
		\label{lem5.3}
		There exists a constant $b_\a>0$ such that
		\[
		\a_k\ge b_\a\D l_k^{\rho_{k+1}}(d_k)
		\]
		holds for all $k>0$.
	\end{Lemma}
	
	\begin{proof}
		By the convexity of $\|\cdot\|_1$ and $[\l_1(\cdot)]_+$, as well as the boundedness of $d_k$, we can deduce that
		\begin{align*}
			&\quad v(x_k+\a d_k)\\
			&=\|h(x_k+\a d_k)\|_1+[\l_1(G(x_k+\a d_k))]_+\\
			&\le\|h_k+\a Dh(x_k)d_k\|_1+[\l_1(G_k+\a DG(x_k)d_k)]_++O(\a^2)\\
			&\le(1-\a)\|h_k\|_1+\a\|h_k+Dh(x_k)d_k\|_1\\
			&\quad+(1-\a)[\l_1(G_k)]_++\a[\l_1(G_k+DG(x_k)d_k)]_++O(\a^2)\\
			&=(1-\a)v_k+\a l_k^v(d_k)+O(\a^2).
		\end{align*}
		Continuing, we have that
		\begin{align*}
			&\quad P^{\rho_{k+1}}(x_k+\a d_k)\\
			&=\rho_{k+1}f(x_k+\a d_k)+v(x_k+\a d_k)\\
			&\le\rho_{k+1}(f_k+\a g_k^Td_k)+(1-\a)v_k+\a l_k^v(d_k)+O(\a^2)\\
			&=\rho_{k+1}f_k+v_k-\a(v_k-\rho_{k+1}g_k^Td_k-l_k^v(d_k))+O(\a^2)\\
			&=P^{\rho_{k+1}}(x_k)-\a\D l_k^{\rho_{k+1}}(d_k)+O(\a^2).
		\end{align*}
		Thus, according to Assumption A1 and A2, there exists a constant $\tau>0$ such that
		\begin{align*}
			&\quad P^{\rho_{k+1}}(x_k+\a d_k)-P^{\rho_{k+1}}(x_k)\\
			&\le-\a\D l_k^{\rho_{k+1}}(d_k)+\tau\a^2\\
			&=-\eta\a\D l_k^{\rho_{k+1}}(d_k)+\tau\a^2-(1-\eta)\a\D l_k^{\rho_{k+1}}(d_k)\\
			&\le-\eta\a\D l_k^{\rho_{k+1}}(d_k)
		\end{align*}
		holds for $\a$ satisfying
		\[
		0<\a\le\dfrac{(1-\eta)}{\tau}\D l_k^{\rho_{k+1}}(d_k).
		\]
		Then, by the line search strategy, it follows that
		\[
		\a_k\ge\dfrac{\gamma(1-\eta)}{\tau}\D l_k^{\rho_{k+1}}(d_k)\ge b_\a\D l_k^{\rho_{k+1}}(d_k),
		\quad
		b_\a:=\dfrac{\gamma(1-\eta)}{\tau}.
		\]
	\end{proof}
	
	Since it is not convenient to describe the decrease of the penalty function from $P^{\rho_k}(x_k)$ to $P^{\rho_{k+1}}(x_{k+1})$, we here introduce the shifted penalty function (as defined in Equation (4.5) in \cite{Burke14})
	\be
	\label{eq5.1}
	\phi(x,\rho):=\rho(f(x)-f_{\min})+v(x)
	\ee
	where $f_{\min}:=\inf_{x\in\Omega}f(x)$. The shifted penalty function $\phi(x,\rho)$ possesses a useful
	monotonicity property established in the following lemma.
	
	\begin{Lemma}
		\label{lem5.4}
		For all $k$,
		\[
		\phi(x_{k+1},\rho_{k+2})-\phi(x_k,\rho_{k+1})\le-\eta\a_k\D l_k^{\rho_{k+1}}(d_k),
		\]
		so, the sequence $\{\phi(x_k,\rho_{k+1})\}$ is monotonically decreasing.
	\end{Lemma}
	
	\begin{proof}
		By the definition of $P^\rho(x)$ and $\phi(x,\rho)$ (recall equations \eqref{eq2.4} and \eqref{eq5.1}),
		we have that
		\[
		\phi(x_{k+1},\rho_{k+1})-\phi(x_{k},\rho_{k+1})
		=P^{\rho_{k+1}}(x_{k+1})-P^{\rho_{k+1}}(x_{k}),
		\]
		and it follows from \eqref{eq3.7} that
		\[
		\phi(x_{k+1},\rho_{k+1})-\phi(x_{k},\rho_{k+1})\le-\eta\a_k\D l_k^{\rho_{k+1}}(d_k).
		\]
		Moreover, since $\rho_{k+2}\le\rho_{k+1}$, $f_{k+1}-f_{\min}\ge0$, we have that
		\begin{align*}
			\phi(x_{k+1},\rho_{k+2})
			&=\rho_{k+2}(f_{k+1}-f_{\min})+v_{k+1}\\
			&\le\rho_{k+1}(f_{k+1}-f_{\min})+v_{k+1}\\
			&=\phi(x_{k+1},\rho_{k+1}).
		\end{align*}
		This allows us to conclude that
		\[
		\phi(x_{k+1},\rho_{k+2})-\phi(x_{k},\rho_{k+1})\le-\eta\a_k\D l_k^{\rho_{k+1}}(d_k).
		\]
		Finally, as $\D l_k^{\rho_{k+1}}(d_k)\ge0$, we can establish that the sequence $\{\phi(x_k,\rho_{k+1})\}$ decreases monotonically.
	\end{proof}
	
	The following two lemmas illustrate that the reductions in $l_k^{\rho_{k+1}}(d_k)$, $l_k^v(d_k)$,
	$l_k^v(d_k^{fea})$, and the norms of $d_k^{fea}$ and $d_k$ tend to zero in the limit.
	
	\begin{Lemma}
		\label{lem5.5}
		The following limits hold:
		\[
		\lim_{k\to\infty}\D l_k^{\rho_{k+1}}(d_k)=\lim_{k\to\infty}\D l_k^v(d_k)=\lim_{k\to\infty}\D
		l_k^v(d_k^{fea})=0.
		\]
	\end{Lemma}
	
	\begin{proof}
		We suppose, by contradiction, that $\D l_k^{\rho_{k+1}}(d_k)$ does not converge to $0$. Then,
		there exist a constant $\tau>0$ and an infinite index set $\m{K}$ such that
		\[
		\lvert\D l_k^{\rho_{k+1}}(d_k)\rvert\ge\tau, \quad k\in\m{K}.
		\]
		Then, by Lemma \ref{lem5.3} and Lemma \ref{lem5.4}, this would implies that
		\[
		\phi(x_{k+1},\rho_{k+2})-\phi(x_k,\rho_{k+1})\le-\eta b_\a\tau^2, \quad k\in\m{K}.
		\]
		Since
		\[
		\phi(x_{k+1},\rho_{k+2})-\phi(x_k,\rho_{k+1})\le0, \quad k\notin\m{K},
		\]
		we deduce that
		\[
		\lim_{k\to\infty}\phi(x_k,\rho_{k+1})\le\phi(x_0,\rho_1)-\sum_{k\in\m{K}}\eta b_\a\tau^2=-\infty,
		\]
		which is impossible since $\phi(x_k,\rho_{k+1})\ge0$ holds for all $k$. Hence, we can conclude that
		\be
		\label{eq5.2}
		\lim_{k\to\infty}\D l_k^{\rho_{k+1}}(d_k)=0,
		\ee
		and remaining proof follows from \eqref{eq4.6} and \eqref{eq5.2}.
	\end{proof}
	
	\begin{Lemma}
		\label{lem5.6}
		The following limits hold:
		\[
		\lim_{k\to\infty}d_k^{fea}=\lim_{k\to\infty}d_k=0.
		\]
	\end{Lemma}
	
	\begin{proof}
		First, we suppose by contradiction that $\lim_{k\to\infty}d_k^{fea}\neq0$. This implies that there exist
		a constant $\tau>0$ and an infinite index set $\m{K}$ such that
		\[
		\|d_k^{fea}\|\ge\tau, \quad k\in\m{K}.
		\]
		By Lemma \ref{lem5.5}, there exists an index $k_0$ such that
		\[
		\D l_k^v(d_k^{fea})\le\dfrac{1}{4}b_1\tau^2, \quad k\ge k_0, \quad k\in\m{K}.
		\]
		Then, for $k\in\m{K}$, we have
		\begin{align*}
			&\quad l_k^v(d_k^{fea})+\dfrac{1}{2}(d_k^{fea})^TB_k^{fea}d_k^{fea}\\
			&=l_k^v(0)-\D l_k^v(d_k^{fea})+\dfrac{1}{2}(d_k^{fea})^TB_k^{fea}d_k^{fea}\\
			&\ge l_k^v(0)-\dfrac{1}{4}b_1\tau^2+\dfrac{1}{2}b_1\tau^2\\
			&>l_k^v(0),
		\end{align*}
		which is a contradiction. This is because $0$ is a feasible point for the problem \eqref{eq3.1}, while $d_k^{fea}\neq0$ for $k\in\m{K}$ is a global minimizer of it. Consequently, we conclude that $\lim_{k\to\infty}d_k^{fea}=0$.
		
		Let's proceed to prove that $\lim_{k\to\infty}d_k=0$. To do this, we first establish that
		\be
		\label{eq5.3}
		\lim_{k\to\infty}\rho_kg_k^Td_k=0.
		\ee
		If $\lim_{k\to\infty}\rho_k=0$, then \eqref{eq5.3} follows from the boundedness of the sequences $\{g_k\}$ and $\{d_k\}$.
		
		Now, consider the case where $\lim_{k\to\infty}\rho_k>0$. By the update rule of $\rho$ (recall \eqref{eq3.5} and \eqref{eq3.6}), for all
		$k$ sufficiently large, we must have that
		\[
		\rho_{k+1}=\rho'_k=\rho_k.
		\]
		Then, by Lemma \ref{lem5.5}, we have
		\[
		0=\lim_{k\to\infty}(\D l_k^v(d_k)-\D l_k^{\rho_{k+1}}(d_k))
		=\lim_{k\to\infty}\rho_{k+1}g_k^Td_k
		=\lim_{k\to\infty}\rho_kg_k^Td_k.
		\]
		We then suppose, by contradiction, that there exist a constant $\tau>0$ and an infinite index set $\m{K}$ such that
		\[
		\|d_k\|\ge\tau, \quad k\in\m{K}.
		\]
		Then, $d_k^TB_kd_k\ge b_2\tau^2$. Moreover, by $\lim_{k\to\infty}d_k^{fea}=0$ and \eqref{eq5.3}, there exists an index $k_0$ such that
		\[
		\rho_kg_k^Td_k^{fea}<\dfrac{1}{8}b_1\tau^2,
		\quad \|d_k^{fea}\|<\dfrac{b_1\tau^2}{4b_2},
		\quad \rho_kg_k^Td_k>-\dfrac{1}{4}b_1\tau^2
		\]
		hold for $k\ge k_0$, $k\in\m{K}$. Hence, we have
		\[
		\rho_kg_k^Td_k^{fea}+\dfrac{1}{2}(d_k^{fea})^TB_kd_k^{fea}
		<\dfrac{1}{8}b_1\tau^2+\dfrac{1}{2}b_2\|d_k^{fea}\|^2
		<\dfrac{1}{4}b_1\tau^2.
		\]
		This implies that
		\[
		\rho_kg_k^Td_k^{fea}+\dfrac{1}{2}(d_k^{fea})^TB_kd_k^{fea}
		<\dfrac{1}{4}b_2\tau^2
		=-\dfrac{1}{4}b_2\tau^2+\dfrac{1}{2}b_2\tau^2
		<\rho_kg_k^Td_k+\dfrac{1}{2}d_k^TB_kd_k,
		\]
		which contradicts the fact that $d_k$ is the global minimizer for \eqref{eq3.3}, while
		$d_k^{fea}$ is a feasible point for it. Hence, we conclude that $\lim_{k\to\infty}d_k=0$.
	\end{proof}
	
	Before giving the results on global convergence, we establish the feasibility for limit points of the sequence $\{x_k\}$ under the case that $\lim_{k\to\infty}\rho_k=0$.
	
	\begin{Lemma}
		\label{lem5.7}
		Let
		\[
		\m{K}_\rho:=\{ k \mid \rho_{k+1}<\rho_k\}
		\]
		be a subset of the iterations during which $\rho_k$ was decreased. In the case that
		$\lim_{k\to\infty}\rho_k=0$, the set of accumulation points for the sequence $\{x_k\}$ is either exclusively constituted of feasible points or entirely comprised of infeasible ones.
	\end{Lemma}
	
	\begin{proof}
		Here we suppose, by contradiction, that there exist two infinite index sets $\m{K}^{fea}$ and
		$\m{K}^{opt}$ such that
		\[
		\lim_{k\to\infty}x_k:=x^{fea}, \quad \lim_{k\to\infty}v_k:=v^{fea}>0, \quad k\in\m{K}^{fea},
		\]
		\[
		\lim_{k\to\infty}x_k:=x^{opt}, \quad \lim_{k\to\infty}v_k=0, \quad k\in\m{K}^{opt}.
		\]
		Proceeding, we analyze the situation in two distinct aspects. Firstly, for indices $k$ belonging to the set $\m{K}^{fea}$ that are sufficiently large, we have that
		\[
		\rho_{k+1}(f_k-f_{\min})\ge0, \quad v_k\ge\dfrac{1}{2}v^{fea},
		\]
		which implies that $\phi(x_k,\rho_{k+1})\ge0.5v^{fea}$, $k\in\m{K}^{fea}$. Secondly, focusing on indices $k$ within the set $\m{K}^{opt}$ that are sufficiently large, since $\lim_{k\to\infty}\rho_k=0$, we have that
		\[
		\rho_{k+1}(f_k-f_{\min})<\dfrac{1}{4}v^{fea}, \quad v_k<\dfrac{1}{4}v^{fea},
		\]
		which implies that $\phi(x_k,\rho_{k+1})<0.5v^{fea}$, $k\in\m{K}^{opt}$. Since both $\m{K}^{fea}$ and
		$\m{K}^{opt}$ are infinite sets, we can deduce the existence of an index pair $k_1\in\m{K}^{fea}$, $k_2\in\m{K}^{opt}$, with $k_1>k_2$ such that
		\[
		\phi(x_{k_1},\rho_{k_1+1})\ge0.5v^{fea}>\phi(x_{k_2},\rho_{k_2+1}),
		\]
		this contradicts that the sequence $\{\phi(x_k,\rho_{k+1})\}$ decreases monotonically. In conclusion, all limit points of the sequence $\{x_k\}$ must either belong to the feasible set or to the infeasible set.
	\end{proof}
	
	Now we present our first theorem within this section, which states that every limit point of an
	infinite sequence generated by Algorithm \ref{alg3.1} is endowed with first-order optimality for the problem \eqref{eq2.1}.
	
	\begin{Theorem}
		\label{thm5.8}
		The following limit holds:
		\[
		\lim_{k\to\infty}R_k^{fea}=0.
		\]
		Therefore, all limit points of $\{(x_k,\bar\mu_{k+1},\bar Y_{k+1})\}$ are first-order optimal
		for \eqref{eq2.1}.
	\end{Theorem}
	
	\begin{proof}
		The proof is obtained by employing the definition of $R_k^{fea}$ as presented in Section \ref{sect3}, coupled with the utilization of \eqref{eq4.3} and $\lim_{k\to\infty}d_k^{fea}=0$.
	\end{proof}
	
	We proceed to establish the proof that under the condition of the penalty parameter $\rho_k$ remaining consistently separated from zero, every accumulation point within the feasible set of the sequence $\{x_k\}$ corresponds to a KKT point.
	
	\begin{Theorem}
		\label{thm5.9}
		If $\lim_{k\to\infty}\rho_k=\rho^*>0$ and $\lim_{k\to\infty}v_k=0$, then
		$\lim_{k\to\infty}R_k^{opt}=0$.
		Thus, every limit point $(x^*,\hat\mu^*/\rho^*,\hat Y^*/\rho^*)$ of the sequence
		$\{(x_k,\hat\mu_{k+1}/\rho_k,\hat Y_{k+1}/\rho_k)\}$ is a KKT point for the problem \eqref{eq1.1}.
	\end{Theorem}
	
	\begin{proof}
		To begin, we establish that the sequence $\{(\hat\mu_{k+1},\hat Y_{k+1})\}$ is bounded. Indeed, if the
		sequence $\{(\hat\mu_{k+1},\hat Y_{k+1})\}$ were unbounded, since
		$\lim_{k\to\infty}\rho_k=\rho^*>0$, then there exists an infinite index set $\m{K}$ such that
		\eqref{eq3.4} holds for $k\in\m{K}$ sufficiently large. Consequently, it follows from \eqref{eq3.5} that
		$\rho_{k+1}\le\rho_k'<\rho_k$ holds for $k\in\m{K}$ sufficiently large, and thus that
		$\lim_{k\to\infty}\rho_k=0$, which presents a contradiction. Thus, the sequence $\{(\hat\mu_{k+1},\hat Y_{k+1})\}$ is
		bounded.
		
		Considering \eqref{eq4.5}, we deduce that
		\[
		\lim_{k\to\infty}\na_xF(x_k,1,\hat\mu_{k+1}/\rho_k,\hat
		Y_{k+1}/\rho_k)=-\lim_{k\to\infty}\dfrac{B_kd_k}{\rho_k}=0.
		\]
		Then, it only remains to demonstrate that
		\[
		\lim_{k\to\infty}\langle\hat Y_{k+1},G_k\rangle=0.
		\]
		By the assumption that $\lim_{k\to\infty}v_k=0$, and referencing Lemma \ref{lem5.5} alongside
		\[
		\D l_k^v(d_k)=l_k^v(0)-l_k^v(d_k)=v_k-(e^T(r_k+s_k)+t_k),
		\]
		we have that
		\[
		\lim_{k\to\infty}\|r_k\|=\lim_{k\to\infty}\|s_k\|=\lim_{k\to\infty}t_k=0.
		\]
		The remainder of the theorem follows from the above limit properties and the complementarity condition of the problem \eqref{eq3.2}.
	\end{proof}
	
	We proceed to establish that if the penalty parameter tends to zero, then every feasible limit point of the sequence
	$\{x_k\}$ corresponds to an FJ point, where the Robinson's constraint qualification is not satisfied.
	
	\begin{Theorem}
		\label{thm5.10}
		Suppose that $\lim_{k\to\infty}\rho_k=0$ and $\lim_{k\to\infty}v_k=0$, with $\m{K}_\rho$ being the index set as defined in Lemma \ref{lem5.7}. Then, all accumulation points of $\{x_k\}_{\m{K}_\rho}$ are indicative of FJ points for the problem \eqref{eq1.1}, where the Robinson's constraint qualification fails.
	\end{Theorem}
	
	\begin{proof}
		Suppose that there exists an infinite index set $\m{K}\subseteq\m{K}_\rho$ such that
		\[
		\lim_{k\in\m{K}}x_k=x^*, \quad \lim_{k\in\m{K}}v_k=v(x^*)=0.
		\]
		Firstly, we establish that
		\be
		\label{eq5.4}
		\|\bar\mu_{k+1}\|_\infty+\text{tr}(\bar Y_{k+1})>1-\e \quad \text{or} \quad
		\|\hat\mu_{k+1}\|_\infty+\text{tr}(\hat Y_{k+1})>1-\e
		\ee
		holds for $k\in\m{K}$ sufficiently large, where $\e>0$ is a constant as defined below the equation \eqref{eq3.6}. To do this, we suppose, by contradiction, that there exists an infinite index set $\m{K}_\e\subseteq\m{K}$ such that
		\be
		\label{eq5.5}
		\|\bar\mu_{k+1}\|_\infty+\text{tr}(\bar Y_{k+1})\le1-\e, \quad
		\|\hat\mu_{k+1}\|_\infty+\text{tr}(\hat Y_{k+1})\le1-\e.
		\ee
		Then, it follows from $\lim_{k\to\infty}\rho_k=0$ that \eqref{eq3.4} is not fulfilled for $k\in\m{K}_\e$ sufficiently
		large. Consequently, in accordance with \eqref{eq3.5}, we have that $\rho_k'=\rho_k$ for $k\in\m{K}_\e$. Moreover, it follows from the first equation in \eqref{eq5.5} that 
		\[
			\|e_l-\bar\mu_{k+1}\|_\infty>0, \quad \lvert1-\text{tr}(\bar Y_{k+1})\rvert>0,
		\]
		and thus by lines 2-4 in \eqref{eq4.3} that $(r_k,s_k,t_k)=0$, which implies that
		\[
		\D l_k^v(d_k)=\D l_k^v(d_k^{fea})=v_k, \quad k\in\m{K}_\e.
		\]
		Then, by the definition of $\D l_k^\rho(d)$ and \eqref{eq4.4},
		\begin{align*}
			\D l_k^{\rho_k'}(d_k)
			&=\D l_k^{\rho_k}(d_k)\\
			&\ge\D l_k^{\rho_k}(d_k)-d_k^TB_kd_k\\
			&=\D l_k^v(d_k)-\rho_kg_k^Td_k-d_k^TB_kd_k\\
			&=v_k+d_k^TDh(x_k)^T\hat\mu_{k+1}+d_k^TDG(x_k)^*\hat Y_{k+1}\\
			&=\|h_k\|_1-\hat\mu_{k+1}^Th_k+[\l_1(G_k)]_+-\langle\hat Y_{k+1},G_k\rangle\\
			&\ge\|h_k\|_1-\|\hat\mu_{k+1}\|_\infty\|h_k\|_1+[\l_1(G_k)]_+-\text{tr}(\hat Y_{k+1})[\l_1(G_k)]_+.
		\end{align*}
		Furthermore, it follows from the second equation in \eqref{eq5.5} that
		\begin{align*}
			\D l_k^{\rho_k'}(d_k)
			&=(1-\|\hat\mu_{k+1}\|_\infty)\|h_k\|_1+(1-\text{tr}(\hat Y_{k+1}))[\l_1(G_k)]_+\\
			&\ge\e(\|h_k\|_1+[\l_1(G_k)]_+)\\
			&=\e\D l_k^v(d_k),
		\end{align*}
		meaning that $\rho_{k+1}$ will not be reduced by \eqref{eq3.6} and thus that
		\[
		\rho_{k+1}=\rho'_k=\rho_k, \quad k\in\m{K}_\e,
		\]
		which contradicts $\m{K}_\e\subseteq\m{K}_\rho=\{k\mid\rho_{k+1}<\rho_k\}$. Hence, \eqref{eq5.4} holds for $k\in\m{K}$ sufficiently
		large.
		
		Then, we will split \eqref{eq5.4} into two cases.
		
		(a) If the first equation of \eqref{eq5.4} holds, we let
		$\e_{k+1}:=\|\bar\mu_{k+1}\|_\infty+\text{tr}(\bar Y_{k+1})$, and
		\[
		\tilde\rho_k:=\dfrac{\rho_k}{\e_{k+1}}, \quad
		\tilde\mu_{k+1}:=\dfrac{\bar\mu_{k+1}}{\e_{k+1}}, \quad
		\tilde Y_{k+1}:=\dfrac{\bar Y_{k+1}}{\e_{k+1}}.
		\]
		Then $(\tilde\rho_k,\tilde\mu_{k+1},\tilde Y_{k+1})$ is bounded and there exists an infinite index set $\m{\tilde K}\subseteq\m{K}$ such that
		\[
		\lim_{k\in\m{\tilde K}}x_k=x^*, \quad
		\lim_{k\in\m{\tilde K}}\tilde\rho_k=0, \quad
		\lim_{k\in\m{\tilde K}}\tilde\mu_{k+1}=\tilde\mu^*, \quad
		\lim_{k\in\m{\tilde K}}\tilde Y_{k+1}=\tilde Y^*.
		\]
		Since that
		\[
		\na_xF(x_k,0,\tilde\mu_{k+1},\tilde Y_{k+1})=\dfrac{Dh(x_k)^T\bar\mu_{k+1}+DG(x_k)^*\bar Y_{k+1}}{\e_{k+1}}=-\dfrac{B_k^{fea}d_k^{fea}}{\e_{k+1}},
		\]
		\[
		\lim_{k\in\m{\tilde K}}\e_{k+1}\ge1-\e>0, \quad \lim_{k\in\m{\tilde K}}d_k^{fea}=0,
		\]
		we have that
		\[
		\na_xF(x^*,0,\tilde\mu^*,\tilde Y^*)
		=\lim_{k\in\m{\tilde K}}\na_xF(x_k,0,\tilde\mu_{k+1},\tilde Y_{k+1})
		=-\lim_{k\in\m{\tilde K}}\dfrac{B_k^{fea}d_k^{fea}}{\e_{k+1}}
		=0.
		\]
		Moreover, it follows from $\lim_{k\to\infty}t_k=0$ and $\langle\bar
		Y_{k+1},G_k+DG(x_k)d_k^{fea}-t_kI_m\rangle=0$ that
		\[
		\langle\tilde Y^*,G(x^*)\rangle=0.
		\]
		
		(b) If the first equation of \eqref{eq5.4} fails, we let
		$\e_{k+1}:=\|\hat\mu_{k+1}\|_\infty+\text{tr}(\hat Y_{k+1})$, and
		\[
		\tilde\rho_k:=\dfrac{\rho_k}{\e_{k+1}}, \quad
		\tilde\mu_{k+1}:=\dfrac{\hat\mu_{k+1}}{\e_{k+1}}, \quad
		\tilde Y_{k+1}:=\dfrac{\hat Y_{k+1}}{\e_{k+1}}.
		\]
		Then $(\tilde\rho_k,\tilde\mu_{k+1},\tilde Y_{k+1})$ is bounded and there exists an infinite index set $\m{\tilde
			K}\subseteq\m{K}$ such that the following limits hold:
		\[
		\lim_{k\in\m{\tilde K}}x_k=x^*, \quad
		\lim_{k\in\m{\tilde K}}\tilde\rho_k=0, \quad
		\lim_{k\in\m{\tilde K}}\tilde\mu_{k+1}=\tilde\mu^*, \quad
		\lim_{k\in\m{\tilde K}}\tilde Y_{k+1}=\tilde Y^*.
		\]
		Furthermore, we have the following relationships:
		\[
		\na_xF(x^*,0,\tilde\mu^*,\tilde Y^*)=\lim_{k\in\m{\tilde
				K}}\na_xF(x_k,\tilde\rho_k,\tilde\mu_{k+1},\tilde Y_{k+1})=-\lim_{k\in\m{\tilde
				K}}\dfrac{B_kd_k}{\e_{k+1}}=0,
		\]
		\[
		\langle\tilde Y^*,G(x^*)\rangle
		=\lim_{k\in\m{\tilde K}}\dfrac{\langle\hat Y_{k+1},G_k+DG(x_k)d_k-t_kI_m\rangle}{\e_{k+1}}
		=0.
		\]
		Hence, $(x^*,0,\tilde\mu^*,\tilde Y^*)$ serves as an FJ point for the problem \eqref{eq1.1}.
		
		Suppose, by contradiction, that the Robinson's constraint qualification holds at $x^*$. Since
		we have previously established that $(x^*,0,\tilde\mu^*,\tilde Y^*)$ serves as an FJ point for problem
		\eqref{eq1.1}, it follows that $(\tilde\mu^*,\tilde Y^*)\ne0$, and thus the following holds:
		\be
		\label{eq5.6}
		\na_xF(x^*,0,\tilde\mu^*,\tilde Y^*)=Dh(x^*)^T\tilde\mu^*+DG(x^*)^*\tilde Y^*=0.
		\ee
		If $\tilde Y^*=0$, then we would have
		\[
		0=Dh(x^*)^T\tilde\mu^*+DG(x^*)^*\tilde Y^*=Dh(x^*)^T\tilde\mu^*.
		\]
		By Definition \ref{def5.1}, $\tilde\mu^*=0$, which contradicts the fact that
		$(\tilde\mu^*,\tilde Y^*)\ne0$. Therefore, $0\ne\tilde Y^*\succeq0$ and thus that
		\[
		\langle \tilde Y^*,G(x^*)+DG(x^*)\bar d\rangle<0,
		\]
		where $\bar d$ is defined in Definition \ref{def5.1}. Moreover, by multiplying $\bar d$ on
		both sides of \eqref{eq5.6}, we have that
		\[
		0=\langle\tilde\mu^*,Dh(x^*)\bar d\rangle+\langle\tilde Y^*,DG(x^*)\bar d\rangle=\langle\tilde
		Y^*,DG(x^*)\bar d\rangle<-\langle\tilde Y^*,G(x^*)\rangle=0,
		\]
		which is a contradiction. Thus, it is concluded that $(x^*,0,\tilde\mu^*,\tilde Y^*)$ constitutes an FJ point for problem \eqref{eq1.1} where the Robinson's constraint qualification fails.
	\end{proof}
	
	We proceed to analyze the global convergence to an infeasible stationary point for \eqref{eq1.1}. We will conclude that such an infeasible stationary point indeed corresponds to an FJ point for the shifted problem \eqref{eq1.2}. The Robinson's constraint qualification for \eqref{eq1.2} is described as follows.
	
	\begin{Lemma}
		\label{lem5.11}
		Suppose that $x^*$ is a feasible point of \eqref{eq1.2} satisfying the Robinson's constraint qualification, then $Dh(x^*)$ has full row rank and
		\[
		Dh(x^*)\bar d=0, \quad G(x^*)+DG(x^*)\bar d\prec t^*I_m,
		\]
		where $\bar d\in\m{R}^n$ is defined in Definition \ref{def5.1}.
	\end{Lemma}
	
	We now establish the proof that under the condition of the penalty parameter remains separated from zero, every infeasible limit point arising from the sequence $\{x_k\}$ corresponds to an infeasible stationary point for problem \eqref{eq1.1}. Moreover, this infeasible stationary point also aligns as a KKT point for  \eqref{eq1.2}.
	
	\begin{Theorem}
		\label{thm5.12}
		Suppose that $\lim_{k\to\infty}\rho_k=\rho^*>0$ and $\lim_{k\to\infty}v_k>0$. Then, all accumulation points of the sequence $\{x_k\}$ correspond to an infeasible stationary point of the problem \eqref{eq1.1} and a KKT point for \eqref{eq1.2}.
	\end{Theorem}
	
	\begin{proof}
		By Theorem \ref{thm5.8} and $\lim_{k\to\infty}v_k>0$, it becomes evident that every accumulation point $x^*$ derived from the sequence $\{x_k\}$ is an infeasible stationary point for \eqref{eq1.1}.
		
		Considering the proof outlined in Theorem \ref{thm5.9}, it has been established that the sequence $\{(\hat\mu_{k+1},\hat Y_{k+1})\}$ is bounded. Let $\m{K}$ be an infinite index set such that
		\[
		\lim_{k\in\m{K}}\bar\mu_{k+1}=\bar\mu^*, \quad
		\lim_{k\in\m{K}}\bar Y_{k+1}=\bar Y^*, \quad
		\lim_{k\in\m{K}}r_k=r^*, \quad
		\lim_{k\in\m{K}}s_k=s^*, \quad
		\lim_{k\in\m{K}}t_k=t^*.
		\]
		Recalling that $\rho_k>0$ and that $(d_k,\hat\mu_{k+1}/\rho_k,\hat Y_{k+1}/\rho_k)$ satisfies the KKT condition of \eqref{eq3.3} as expressed in equation \eqref{eq5.7}, which provides
		\be
		\label{eq5.7}
		\ba{c}
		g_k+B_kd_k/\rho_k+Dh(x_k)^T(\hat\mu_{k+1}/\rho_k)+DG(x_k)^*(\hat Y_{k+1}/\rho_k)=0,\\
		h_k+Dh(x_k)d_k=r_k-s_k,\\
		G_k+DG(x_k)d_k\preceq t_kI_m,\\
		\langle\hat Y_{k+1}/\rho_k,G_k+DG(x_k)d_k-t_kI_m\rangle=0, \quad \hat Y_{k+1}\succeq0,
		\ea
		\ee
		we proceed by taking limits in \eqref{eq5.7} for $k\in\m{K}$ and conclude that
		\[
		\ba{c}
		g(x^*)+Dh(x^*)^T(\hat\mu^*/\rho^*)+DG(x^*)^*(\hat Y^*/\rho^*)=0,\\
		h(x^*)=r^*-s^*,\\
		G(x^*)\preceq t^*I_m,\\
		\langle\hat Y^*/\rho^*,G(x^*)-t^*I_m\rangle=0, \quad \hat Y^*\succeq0.
		\ea
		\]
		Consequently, it is established that $(x^*,\hat\mu^*/\rho^*,\hat Y^*/\rho^*)$ represents a KKT point for \eqref{eq1.2}.
	\end{proof}
	
	We demonstrate that if the penalty parameter $\rho_k$ tends to zero, every infeasible accumulation point derived from the sequence $\{x_k\}$ corresponds to an infeasible stationary point for \eqref{eq1.1}. Additionally, these accumulation points also correspond to FJ points for \eqref{eq1.2} in instances where the Robinson's constraint qualification fails.
	
	\begin{Theorem}
		\label{thm5.13}
		Suppose that $\lim_{k\to\infty}\rho_k=0$ and $\lim_{k\to\infty}v_k>0$, $\m{K}_\rho$ is the index set defined in Lemma \ref{lem5.7}. Then, all accumulation points of $\{x_k\}_{\m{K}_\rho}$ correspond to infeasible stationary points for \eqref{eq1.1}. Furthermore, these accumulation points also represent FJ points for \eqref{eq1.2} in cases where the Robinson's constraint qualification fails.
	\end{Theorem}
	
	\begin{proof}
		As in the proof of Theorem \ref{thm5.12}, $x^*$ is an infeasible stationary point for \eqref{eq1.1}.
		Moreover, by Lemma \ref{lem5.6} and the constraints in \eqref{eq3.2}, we deduce that $(r_k,s_k,t_k)\ne0$ holds for
		$k$ sufficiently large. Thus, by referring to \eqref{eq4.3},
		\[
		\|\bar\mu_{k+1}\|_\infty=1 \quad \text{or} \quad \l_1(\bar Y_{k+1})=1
		\]
		holds for $k$ sufficiently large, which implies that $(\bar\mu_{k+1},\bar Y_{k+1})\ne0$. Since $(\bar\mu_{k+1},\bar Y_{k+1})$ is bounded, let us assume that there exists an infinite index set $\m{K}\subseteq\m{K}_\rho$ such that
		\[
		\lim_{k\in\m{K}}x_k=x^*, \quad
		\lim_{k\in\m{K}}\rho_k=0, \quad
		\lim_{k\in\m{K}}\bar\mu_{k+1}=\bar\mu^*, \quad
		\lim_{k\in\m{K}}\bar Y_{k+1}=\bar Y^*.
		\]
		Then, we have that
		\[
		\na_xF(x^*,0,\bar\mu^*,\bar Y^*)=\lim_{k\in\m{K}}\na_xF(x_k,0,\bar\mu_{k+1},\bar
		Y_{k+1})=-\lim_{k\in\m{K}}B_k^{fea}d_k^{fea}=0.
		\]
		Furthermore, it follows from $\lim_{k\to\infty}t_k=t^*$ and $\langle\bar Y_{k+1},G_k+DG(x_k)d_k^{fea}-t_kI_m\rangle=0$ that
		\[
		\langle\bar Y^*,G(x^*)-t^*I_m\rangle=0.
		\]
		Hence, one can deduce that $(x^*,0,\bar\mu^*,\bar Y^*)$ represents an FJ point for \eqref{eq1.2}.
		
		Suppose, by contradiction, that the Robinson's constraint qualification holds at $x^*$. Having already established that $(x^*,0,\bar\mu^*,\bar Y^*)$ represents an FJ point for \eqref{eq1.2}, it follows that $(\bar\mu^*,\bar Y^*)\ne0$. This leads that
		\be
		\label{eq5.8}
		\na_xF(x^*,0,\bar\mu^*,\bar Y^*)=Dh(x^*)^T\bar\mu^*+DG(x^*)^*\bar Y^*=0.
		\ee
		In the event that $\bar Y^*=0$, we have that
		\[
		0=Dh(x^*)^T\bar\mu^*+DG(x^*)^*\bar Y^*=Dh(x^*)^T\bar\mu^*.
		\]
		By Lemma \ref{lem5.11}, $\bar\mu^*=0$. However, this contradicts the fact that $(\bar\mu^*,\bar
		Y^*)\ne0$. Therefore, $0\ne\bar Y^*\succeq0$ and
		\[
		\langle\bar Y^*,G(x^*)+DG(x^*)\bar d-t^*I_m\rangle<0,
		\]
		where $\bar d$ is a parameter mentioned in Lemma \ref{lem5.11}. Moreover, by multiplying $\bar d$ on
		both sides of \eqref{eq5.8}, we have
		\[
		0=\langle\bar\mu^*,Dh(x^*)\bar d\rangle+\langle\bar Y^*,DG(x^*)\bar d\rangle=\langle\bar
		Y^*,DG(x^*)\bar d\rangle<-\langle\bar Y^*,G(x^*)-t^*I_m\rangle=0,
		\]
		which is a contradiction. Thus, we conclude that $(x^*,0,\bar\mu^*,\bar Y^*)$ represents an FJ point for \eqref{eq1.2} where the Robinson's constraint qualification fails.
	\end{proof}
	
	The theorems derived in this section are directly applicable to the limit values of the sequences $\{\rho_k\}$ and $\{v_k\}$. The resulting accumulation point $x^*$ holds a significant position known as the Fritz-John point, satisfying various conditions as outlined in Table \ref{tab5.1}. By referring to the theorems presented in this section, one can gain a deeper understanding of the exact relationship between the limit values of the aforementioned sequences and the various types of accumulation points.
	\begin{table}[h]
		\caption{Summarize of the theorems}\label{tab5.1}%
		\begin{tabular}{@{}llll@{}}
			\toprule
			Theorem & $\lim\rho_k$ & $\lim v_k$ & Results \\
			\midrule
			Theorem \ref{thm5.9} & $>0$ & $=0$ & A KKT point of \eqref{eq1.1}.\\
			Theorem \ref{thm5.10} & $=0$ & $=0$ & Robinson's CQ of \eqref{eq1.1} fails.\\
			Theorem \ref{thm5.12} & $>0$ & $>0$ & An infeasible stationary point of \eqref{eq1.1}, and a KKT point of \eqref{eq1.2}.\\
			Theorem \ref{thm5.13} & $=0$ & $>0$ & An infeasible stationary point of \eqref{eq1.1}, and Robinson's CQ of \eqref{eq1.2} fails.\\
			\botrule
		\end{tabular}
	\end{table}

	\section{Numerical experiments}
	\label{sect6}
	
	Some numerical experiments were done in order to demonstrate the theoretical properties of Algorithm \ref{alg3.1}. We developed an implementation of the algorithm in MATLAB (version R2023a) and tested its performance under several situations. The subproblems \eqref{eq3.2} and \eqref{eq3.3} were both solved by SeDuMi solver (version 1.32) with default settings. The matrix $B_k^{fea}$ was set as a constant one, while $B_k$ was defined as $B_k:=\max\{10^{-5},\rho_k\}B_k^{\rm bfgs}$, with $B_k^{\rm bfgs}$ being updated using the modified BFGS updating formula (\cite{Yamashita12}).
	
	The initial parameters were chosen as follows:
	\[
	B_k^{fea}=0.001I_n, \quad \eta=10^{-4}, \quad \epsilon=10^{-4}, \quad \delta=0.9, \quad \gamma=0.6, \quad \rho_0=1.
	\]
	Algorithm \ref{alg3.1} was designed to terminate under the conditions that $\|d_k\|<10^{-4}$ and either:
	
	(i) $v(x_k)<10^{-4}$, signifying that $x_k$ corresponds to a Fritz-John point.
	
	(ii) $v(x_k)\ge10^{-4}$, indicating that $x_k$ corresponds to an infeasible stationary point.
	
	Several small problems with different situations are tested. Problem \eqref{eq6.1} and Problem \eqref{eq6.2} are instances with no feasible solutions. Problem \eqref{eq6.3} and Problem \eqref{eq6.4} are feasible problems, yet the Robinson's constraint qualification fails at each solution. Problem \eqref{eq6.5} and Problem \eqref{eq6.6} are also feasible problems, but the linearized constraints at the proposed initial points are inconsistent.
	
	The first test problem is generated from the so-called \emph{isolated} problem in \cite{Byrd10}:
	\be
	\label{eq6.1}
	\ba{cl}
	\min & x_1+x_2\\
	\st & \left(\ba{cc}
	-1 & x_1 \\
	x_1&1+x_2
	\ea\right)\preceq0, \quad \left(\ba{cc}
	-1 & x_1 \\
	x_1&1-x_2\\
	\ea\right)\preceq0,\\
	\vspace{-5mm}\\
	& \left(\ba{cc}
	-1 & x_2  \\
	x_2&1+x_1 \\
	\ea\right)\preceq0, \quad \left(\ba{cc}
	-1 & x_2 \\
	x_2&1+x_1\\
	\ea\right)\preceq0.
	\ea
	\ee
	The standard initial point is $x_0=(3,2)$, and its corresponding solution denoted as $x^*=(0,0)$ is a strict minimizer of
	the infeasibility measure \eqref{eq2.1}. Algorithm \ref{alg3.1} terminates at an approximate point that closely approaches the true solution. As evidenced by Table \ref{tab6.1}, the initial value of $\rho_0$ is computed to be $0.1109$, leading to the identification of an infeasible stationary point.
	\begin{table}[h]
		\caption{Output for test problem \eqref{eq6.1}}\label{tab6.1}%
		\begin{tabular}{@{}lllllll@{}}
			\toprule
			$k$ & $\rho_k$ & $x_k$ & $\|d_k\|$ & $l_k^v(d_k^{fea})$ & $v(x_k)$ & $f(x_k)$\\
			\midrule
			0 & 1.0000 & (3.0000e$+$00, 2.0000e$+$00) & 3.6056e$+$00 & 1.0000 & 4.7016 & 5.0000\\
			1 & 0.1109 & (1.8680e$-$10, 1.0992e$-$10) & 2.1690e$-$10 & 1.0000 & 1.0000 & 0.0000\\
			\botrule
		\end{tabular}
	\end{table}
	
	The second example is modified from the \emph{nactive} problem in \cite{Byrd10}:
	\be
	\label{eq6.2}
	\ba{cl}
	\min & x_1 \\
	\st & \left(\ba{cc}
	-1 & x_2 \\
	x_2&0.5(x_1+1)
	\ea\right)\preceq0, \quad \left(\ba{cc}
	-1 & x_2 \\
	x_2&-x_1\\
	\ea\right)\preceq0, \quad x_1-x_2^2\le0.
	\ea
	\ee
	The given initial point is $x_0=(-20,10)$. An infeasible
	stationary point $x^*=(0,0)$ is derived by \cite{Byrd10} under the measure of $\ell_1$-norm shift. Another infeasible stationary point $x^*=(-0.2,0)$ is derived by \cite{Dai20} under the measure of $\ell_2$-norm shift. Algorithm \ref{alg3.1} terminates at an infeasible
	stationary point $x^*=(-0.3333,0)$ under the measure of $\ell_\infty$-norm shift, where $v(x^*)=0.3333$, $f(x^*)=-0.3333$. A more comprehensive overview of the results can be found in Table \ref{tab6.2}.
	\begin{table}[h]
		\caption{Output for test problem \eqref{eq6.2}}\label{tab6.2}%
		\begin{tabular}{@{}lllllll@{}}
			\toprule
			$k$ & $\rho_k$ & $x_k$ & $\|d_k\|$ & $l_k^v(d_k^{fea})$ & $v(x_k)$ & $f(x_k)$\\
			\midrule
			0 & 1.0000 & $(-20.0000,10.0000)$ & 2.0353e$+$01 & 4.4728 & 24.0000 & $-$20.0000\\
			1 & 0.0087 & $(-0.3333,4.7597)$ & 2.6157e$+$00 & 1.9119 & 4.4728 & $-$0.3333\\
			2 & 0.0087 & $(-0.3333,2.1440)$ & 1.3193e$+$00 & 0.7271 & 1.9119 & $-$0.3333\\
			3 & 0.0087 & $(-0.3333,0.8247)$ & 8.1657e$-$01 & 0.3334 & 0.7271 & $-$0.3333\\
			4 & 0.0087 & $(-0.3333,0.0081)$ & 8.1009e$-$03 & 0.3333 & 0.3334 & $-$0.3333\\
			5 & 0.0087 & $(-0.3333,0.0000)$ & 1.7220e$-$05 & 0.3333 & 0.3333 & $-$0.3333\\
			\botrule
		\end{tabular}
	\end{table}

	We consider third problem called \emph{counterexample}, which is taken from Problem TP3 in \cite{Dai20} with a negative semidefinite constraint:
	\be
	\label{eq6.3}
	\ba{cl}
	\min & x_1 \\
	\st & x_1^2-x_2-1=0,\\
	& x_1-x_3-2=0,\\
	& \left(\ba{cc}-x_2&0\\0&-x_3\ea\right)\preceq0.
	\ea
	\ee
	The initial point is $x_0=(-4,1,1)$. This problem has a unique global minimizer $x^*=(2,3,0)$, at which the Robinson's constraint qualification fails. Algorithm \ref{alg3.1} terminates at an approximate solution $(2.00,3.00,0.00)$ in 6 iterations. See Table \ref{tab6.3} for more details.
	\begin{table}[h]
		\caption{Output for test problem \eqref{eq6.3}}\label{tab6.3}%
		\begin{tabular}{@{}lllllll@{}}
			\toprule
			$k$ & $\rho_k$ & $x_k$ & $\|d_k\|$ & $l_k^v(d_k^{fea})$ & $v(x_k)$ & $f(x_k)$\\
			\midrule
			0 & 0.0127 & $(-4.00,1.00,1.00)$ & 7.0000e+00 & 3.6667 & 21.0000 & $-$4.0000\\
			1 & 0.0127 & $(-1.67,-3.67,-3.67)$ & 2.1762e+00 & 2.4103 & 9.1111 & $-$1.6667\\
			2 & 0.0127 & $(-0.41,-2.41,-2.41)$ & 1.5019e+00 & 1.5432 & 3.9888 & $-$0.4103\\
			3 & 0.0127 & $(0.46,-1.54,-1.54)$ & 3.0718e+00 & $-$0.0000 & 2.2950 & 0.4568\\
			4 & 0.0127 & $(1.38,-0.25,-0.62)$ & 2.9951e+00 & $-$0.0000 & 1.7753 & 1.3827\\
			5 & 0.0127 & $(2.00,2.62,0.00)$ & 3.8101e$-$01 & $-$0.0000 & 0.3810 & 2.0000\\
			6 & 0.0114 & $(2.00,3.00,0.00)$ & 9.0896e$-$09 & 0.0000 & 0.0000 & 2.0000\\
			\botrule
		\end{tabular}
	\end{table}

	The fourth \emph{standard} test problem is the one taken from Problem TP4 in \cite{Dai20} with a negative semidefinite constraint:
	\be
	\label{eq6.4}
	\ba{cl}
	\min & (x_1-2)^2+x_2^2 \\
	\st & \left(\ba{ccc}
	-(1-x_1)^3+x_2 & 0 & 0 \\
	0 & -x_1 & 0 \\
	0 & 0 & -x_2
	\ea\right)\preceq0.
	\ea
	\ee
	The initial point $x_0=(-2,-2)$ is an infeasible point. Its optimal solution $x^*=(1,0)$ is not a KKT point but is a singular stationary one at which the Robinson's constraint qualification fails. Numerical results in Table \ref{tab6.4} show that Algorithm \ref{alg3.1} terminates at an approximate point to the solution.
	\begin{table}[h]
		\caption{Output for test problem \eqref{eq6.4}}\label{tab6.4}%
		\begin{tabular}{@{}lllllll@{}}
			\toprule
			$k$ & $\rho_k$ & $x_k$ & $\|d_k\|$ & $l_k^v(d_k^{fea})$ & $v(x_k)$ & $f(x_k)$\\
			\midrule
			0 & 1.0000 & $(-2.0000,-2.0000)$ & 1.5117e+00 & 0.9310 & 2.0000 & 20.0000\\
			1 & 0.0098 & $(-0.9310,-0.9310)$ & 9.7193$e-$01 & 0.2438 & 0.9310 & 9.4578\\
			2 & 0.0098 & $(-0.2438,-0.2438)$ & 5.1973$e-$01 & $-$0.0000 & 0.2438 & 5.0939\\
			3 & 0.0098 & $(0.1503,0.0950)$ & 2.9874$e-$01 & $-$0.0000 & 0.0000 & 3.4303\\
			4 & 0.0098 & $(0.4336,0.0000)$ & 1.8881$e-$01 & $-$0.0000 & 0.0000 & 2.4537\\
			5 & 0.0098 & $(0.6224,0.0000)$ & 1.2587$e-$01 & $-$0.0000 & 0.0000 & 1.8979\\
			6 & 0.0098 & $(0.7482,0.0000)$ & 8.3918$e-$02 & $-$0.0000 & 0.0000 & 1.5669\\
			7 & 0.0098 & $(0.8322,0.0000)$ & 5.5945$e-$02 & $-$0.0000 & 0.0000 & 1.3638\\
			8 & 0.0098 & $(0.8881,0.0000)$ & 3.7296$e-$02 & $-$0.0000 & 0.0000 & 1.2363\\
			9 & 0.0098 & $(0.9254,0.0000)$ & 2.4863$e-$02 & $-$0.0000 & 0.0000 & 1.1548\\
			10 & 0.0098 & $(0.9503,0.0000)$ & 1.6577$e-$02 & $-$0.0000 & 0.0000 & 1.1019\\
			11 & 0.0098 & $(0.9668,0.0000)$ & 1.1050$e-$02 & $-$0.0000 & 0.0000 & 1.0674\\
			12 & 0.0098 & $(0.9779,0.0000)$ & 7.3682$e-$03 & $-$0.0000 & 0.0000 & 1.0447\\
			13 & 0.0098 & $(0.9853,-0.0000)$ & 4.9100$e-$03 & $-$0.0000 & 0.0000 & 1.0297\\
			14 & 0.0098 & $(0.9902,0.0000)$ & 3.2752$e-$03 & $-$0.0000 & 0.0000 & 1.0197\\
			15 & 0.0073 & $(0.9934,0.0000)$ & 2.1836$e-$03 & $-$0.0000 & 0.0000 & 1.0131\\
			16 & 0.0044 & $(0.9956,0.0000)$ & 1.4506$e-$03 & $-$0.0000 & 0.0000 & 1.0088\\
			17 & 0.0033 & $(0.9971,0.0000)$ & 9.6549$e-$04 & $-$0.0000 & 0.0000 & 1.0058\\
			18 & 0.0020 & $(0.9980,0.0000)$ & 6.5014$e-$04 & $-$0.0000 & 0.0000 & 1.0039\\
			19 & 0.0015 & $(0.9987,0.0000)$ & 4.2479$e-$04 & $-$0.0000 & 0.0000 & 1.0026\\
			20 & 0.0009 & $(0.9991,0.0000)$ & 2.9113$e-$04 & $-$0.0000 & 0.0000 & 1.0018\\
			21 & 0.0007 & $(0.9994,0.0000)$ & 1.9559$e-$04 & 0.0000 & 0.0000 & 1.0012\\
			22 & 0.0004 & $(0.9996,-0.0000)$ & 1.3155$e-$04 & 0.0000 & 0.0000 & 1.0008\\
			23 & 0.0003 & $(0.9997,-0.0000)$ & 3.3915$e-$05 & 0.0000 & 0.0000 & 1.0005\\
			\botrule
		\end{tabular}
	\end{table}

	The fifth problem is the following \emph{Rosen-Suzuki} problem (\cite{Charalambous77}):
	\be
	\label{eq6.5}
	\ba{cl}
	\dmin_{x\in\m{R}^4}& f(x)=x_1^2+x_2^2+2x_3^2+x_4^2-5x_1-5x_2-21x_3+7x_4\\
	\st & h(x)=\left(\ba{c}
	x_1^2+x_2^2+x_3^2+x_4^2+x_1-x_2+x_3-x_4-8\\
	x_1^2+2x_2^2+x_3^2+2x_4^2-x_1-x_4-9\\
	2x_1^2+x_2^2+x_3^2-x_2-x_4-5
	\ea\right)=0,\\
	\vspace{-5mm}\\
	& G(x)=\left(\ba{cccc}
	-x_2-x_3&0&0&0\\
	0&2x_4&-x_1&0\\
	0&-x_1&-x_1&0\\
	0&0&0&-x_2-x_3
	\ea\right)\preceq0.
	\ea
	\ee
	The solution is $x^*=(0,1,2,-1)$. Numerical results are listed in Table \ref{tab6.5}. $x_0$ stands for the initial point. Two infeasible stationary points are detected, one is $x^*=(-0.0000,1.3788,2.2799,-0.0000)$ with $v(x^*)=0.7201$, $f(x^*)=-42.4746$ from the initial point $( 1, 1, 1, 1)$,  the other one is $x^*=(-0.0425, -1.1544, 1.1113, -1.3594)$ with $v(x^*)=0.04314$, $f(x^*)=-21.2158$ from the initial point $(-1, -1, -1, -1)$, see the ``*'' line in Table \ref{tab6.5}, the symbol ``{\rm Nit}''  represents for the total number of iterations.
	\begin{table}[h]
		\caption{Output for test problem \eqref{eq6.5}}\label{tab6.5}%
		\begin{tabular}{@{}llllll@{}}
			\toprule
			$x_0$ & ${\rm Nit}$ & $\rho^*$ & $l_k^v(d_*^f)$ & $v(x^*)$ & $f(x^*)$\\
			\midrule
			$(0,0,0,0)$ &  8 & 0.0466 & 2.6390e$-$14 & 9.7674e$-$11 & $-$44.0000\\
			$(1,1,1,1)$ &  4 & 0.0022 & 7.2012e$-$01 & 7.2012e$-$01 & $-$42.4746*\\
			$(-1,-1,-1,-1)$ &  6 & 0.0087 & 4.3135e$-$02 & 4.3135e$-$02 & $-$21.2158*\\
			$(2,2,2,2)$ &  9 & 0.0336 & 6.5271e$-$09 & 3.2027e$-$11 & $-$44.0000\\
			$(-2,-2,-2,-2)$ & 19 & 0.0023 & 2.6883e$-$14 & 3.7372e$-$09 & $-$44.0000\\
			$(3,3,3,3)$ &  7 & 0.0458 & 6.5269e$-$09 & 3.5874e$-$08 & $-$44.0000\\
			$(-3,-3,-3,-3)$ & 12 & 0.0163 & 2.0595e$-$14 & 1.7571e$-$11 & $-$44.0000\\
			$(4,4,4,4)$ &  9 & 0.0645 & 2.2969e$-$14 & 9.2907e$-$11 & $-$44.0000\\
			$(-4,-4,-4,-4)$ & 13 & 0.0200 & 2.2528e$-$14 & 3.7266e$-$09 & $-$44.0000\\
			$(5,5,5,5)$ & 10 & 0.0996 & 1.9654e$-$14 & 2.7729e$-$10 & $-$44.0000\\
			$(-5,-5,-5,-5)$ & 15 & 0.0187 & 4.8459e$-$14 & 2.5442e$-$08 & $-$44.0000\\
			$(10,10,10,10)$ & 10 & 0.1098 & 2.2859e$-$14 & 5.9946e$-$11 & $-$44.0000\\
			$(-10,-10,-10,-10)$ & 19 & 0.0121 & 1.7289e$-$14 & 6.3386e$-$09 & $-$44.0000\\
			$(100,100,100,100)$ & 15 & 0.0040 & 1.6588e$-$14 & 2.9108e$-$10 & $-$44.0000\\
			$(-100,-100,-100,-100)$ & 17 & 0.0004 & 1.9853e$-$14 & 2.1445e$-$09 & $-$44.0000\\
			\botrule
		\end{tabular}
	\end{table}
	
	We consider the \emph{Hock-Shittkowski} problem combined with the positive semidefinite constraint, see Problem 2 in \cite{Wu13}:
	\be
	\label{eq6.6}
	\ba{cl}
	\dmin_{x\in\m{R}^6}& f(x)=x_1x_4(x_1+x_2+x_3)+x_3\\
	\st & h(x)=\left(\ba{c}
	x_1x_2x_3x_4-x_5-25\\
	x_1^2+x_2^2+x_3^2+x_4^2-x_6-40
	\ea\right)=0,\\
	& G(x)=\left(\ba{cccc}
	x_1&x_2&0&0\\
	x_2&x_4&x_2+x_3&0\\
	0&x_2+x_3&x_4&x_3\\
	0&0&x_3&x_1
	\ea\right)\succeq0,\\
	&x\in\left\{x\in\m{R}^6\mid 1\le x_i\le5, i=1,2,3,4, x_j\ge0, j=5,6\right\}.
	\ea
	\ee
	An approximate minimizer $x^*=(2.7586,1.0000,2.5278,5.0000,9.8668,-0.0000)$ is identified for every chosen initial point. Extensive numerical outcomes regarding Problem \eqref{eq6.6} are presented in Table \ref{tab6.6}. Notably, it should be highlighted that the augmented Lagrangian method proposed in \cite{Wu13} yields a comparatively larger minimizer, specifically $f_k=128.8015$. 
	\begin{table}[h]
		\caption{Output for test problem \eqref{eq6.6}}\label{tab6.6}%
		\begin{tabular}{@{}lllll@{}}
			\toprule
			$x_0$ & ${\rm Nit}$ & $\rho^*$ & $v(x^*)$ & $f(x^*)$\\
			\midrule
			(1,1,1,1,1,1) & 15 & 0.0001 & 4.3142e$-$06 & 89.2383\\
			(2,2,2,2,2,2) & 17 & 0.0056 & 6.6705e$-$06 & 89.2382\\
			(3,3,3,3,3,3) & 35 & 0.0012 & 3.3805e$-$06 & 89.2385\\
			(4,4,4,4,4,4) & 17 & 0.0001 & 2.0528e$-$06 & 89.2384\\
			(5,5,5,5,5,5) & 16 & 0.0018 & 4.3712e$-$08 & 89.2384\\
			\botrule
		\end{tabular}
	\end{table}

	\section{Concluding remarks}
	\label{sect7}
	
	In this paper, we have developed a globally convergent SQP-type method with the least constraint violation for nonlinear semidefinite programming. Our approach establishes a crucial connection between infeasible stationary points of the original problem and Fritz-John points of a shifted version of the problem. Leveraging a two-phase strategy, we compute a feasible direction $d_k^{fea}$ to enhance constraint feasibility through linearized constraints. The search direction $d_k$ is derived from a quadratic semidefinite programming problem. We have rigorously established the global convergence of our method to first-order optimal points with the least constraint violation.
	
	The empirical evaluation of Algorithm \ref{alg3.1} on problems exhibiting degeneracy showcases its efficacy in practice. It's important to note that while our method is primarily inspired by nonlinear programming, certain subproblem optimality conditions had to be adapted to accommodate the structure of semidefinite constraints. As a direction for future research, one could explore refining the rate of convergence towards infeasible stationary points by effectively utilizing second-order information. This extension could potentially enhance the efficiency of our approach and provide deeper insights into the theoretical properties of the method.
	
%

\section*{}
\textbf{Data availability}
We do not analyse or generate any datasets, because our work proceeds within a theoretical and mathematical approach.

\section*{Declarations}
\textbf{Conflicts of interest}
The authors declare that they have no conflict of interest.


\end{document}